\documentclass[12pt, reqno, letterpaper]{amsart}
\usepackage{enumitem}

\numberwithin{equation}{section}

\usepackage[bindingoffset=0.2in,%
            left=1in,right=1in,top=1in,bottom=1in,%
            footskip=.25in]{geometry}

\newtheorem{theorem}{Theorem}[section]
\newtheorem{definition}[theorem]{Definition}

\newtheorem{lemma}[theorem]{Lemma}

\newtheorem{remark}[theorem]{Remark}

\def\al{\aligned}
\def\eal{\endaligned}
\def\be{\begin{equation}}
\def\ee{\end{equation}}
\def\lab{\label}
\def\a{\alpha}

\def\e{\epsilon}

\def\M{{\bf M}}

\def\al{\aligned}

\def\pa{\partial}
\def\d{\nabla}

\def\mb{\mathbb}

\DeclareMathOperator{\Ric}{Ric}

\DeclareMathOperator{\vol}{Vol}

\usepackage{color}
\usepackage{esint}
\numberwithin{equation}{section}

\let\oldtocsection=\tocsection

\let\oldtocsubsection=\tocsubsection

\let\oldtocsubsubsection=\tocsubsubsection

\renewcommand{\tocsection}[2]{\hspace{0em}\oldtocsection{#1}{#2}}
\renewcommand{\tocsubsection}[2]{\hspace{2em}\oldtocsubsection{#1}{#2}}
\renewcommand{\tocsubsubsection}[2]{\hspace{2em}\oldtocsubsubsection{#1}{#2}}
\usepackage{soul}

\begin{document}
\setstcolor{red}

\title[]{Bounds on harmonic radius and  limits of manifolds with bounded Bakry-\'Emery Ricci curvature}
\author{Qi S. Zhang and Meng Zhu}
\address{Department of Mathematics, University of California, Riverside, Riverside, CA 92521, USA}
\email{qizhang@math.ucr.edu}
\address{Department of Mathematics, East China Normal University, Shanghai 200241, China - and Department of Mathematics, University of California, Riverside, Riverside, CA 92521, USA}
\email{mzhu@math.ecnu.edu.cn}
\date{}

\begin{abstract}
Under the usual condition that the volume of a geodesic ball is close to the Euclidean one or the injectivity radii is bounded from below, we prove a lower bound of the $C^{\alpha} \cap W^{1, q}$  harmonic radius for manifolds with bounded Bakry-\'Emery Ricci curvature when the gradient of the potential is bounded. Under these conditions, the regularity that can be imposed on the metrics under harmonic coordinates is only $C^\a \cap W^{1,q}$, where $q>2n$ and $n$ is the dimension of the manifolds. This is almost 1 order lower than that in the classical $C^{1,\a} \cap W^{2, p}$ harmonic coordinates under bounded Ricci curvature condition \cite{And}. The loss of regularity induces some difference in the method of proof, which can also be used to  address the detail of $W^{2, p}$ convergence in the classical case.

Based on this lower bound and the techniques in \cite{ChNa2} and \cite{WZ}, we extend Cheeger-Naber's Codimension 4 Theorem in \cite{ChNa2} to the case where the manifolds have  bounded Bakry-\'Emery Ricci curvature when the gradient of the  potential is bounded.  This result covers Ricci solitons when the gradient of the potential is bounded.

During the proof, we will use a Green's function argument and adopt a linear algebra argument in \cite{Bam}. A new ingradient is to show that the diagonal entries of the matrices in the Transformation Theorem are bounded away from 0. Together these seem to simplify the proof of the Codimension 4 Theorem, even in the case where Ricci curvature is bounded.
\end{abstract}
\maketitle

\tableofcontents

\section{Introduction}

In this paper we extend two important results from the case of bounded Ricci curvature to the case of bounded Bakry-\'Emery curvature with $C^1$ potential. One of these is Anderson's lower bound for harmonic radius \cite{And} and the other is Cheeger-Naber's co-dimension 4 theorem \cite{ChNa2}. While many results in these two cases are parallel, extending these two results require some new effort which we explain now.

In a series of works (\cite{Co}, \cite{ChCo1}, \cite{ChCo2}, \cite{ChCo3}, \cite{ChCo4}, \cite{CCT}, \cite{CoNa}, \cite{ChNa1}, \cite{ChNa2}), Cheeger-Colding-Tian-Naber developed a very deep and powerful theory for studying the Gromov-Hausdorff limits of manifolds with bounded Ricci curvature. In particular, when the manifolds are in addition volume noncollapsed, according to their results, we know that
the Gromov-Hausdorff limits decompose into the union of the regular set and singular set. The regular set is an open convex $C^{1,\a}$ manifold, the singular set has codimension at least 4, and the tangent cone at any point must be a metric cone.

However, there are objects in geometry where the boundedness of the Ricci curvature is not available. One of these is a Ricci soliton under the typical condition that the gradient of the potential is bounded.
More generally, these solitons belong to a class of manifolds where the Bakry-\'Emery Ricci curvature is bounded. The later has become a subject of study by numerous authors.
 Many of the classical geometric and analytic results such as volume comparison theorems and gradient bounds, valid under pointwise Ricci bound,  have been extended to this case in the papers \cite{Q}, \cite{xLi} and \cite{WW}. Recently
  F. Wang and X.H. Zhu \cite{WZ} established analogous results for most of the Cheeger-Colding-Tian-Naber theory.  One notable exception is the codimension 4 theorem for the singular part.
A goal of this paper is to prove such a theorem.

Another case of interest is when the Ricci curvature in only in certain $L^p$ spaces (see e.g. \cite{Ya} or \cite{DWZ} for motivation).
The first effort was made by Petersen-Wei \cite{PeWe1} and \cite{PeWe2}, where they assumed that $|Ric^-|\in L^p$ for some $p>n/2$ and obtained extended Laplacian and volume comparison theorems and continuity of volume under Gromov-Hausdorff limit.
Recently, Tian-Z. Zhang \cite{TZz} successfully extended most of the Cheeger-Colding-Tian-Naber theory except for the codimension 4 theorem for the singular part. Bamler \cite{Bam} proves a codimension 4 theorem for some Ricci flat singular spaces.

In proving these results under weaker Ricci curvature conditions, one needs to extend many key ingredients therein, such as Cheng-Yau gradient estimate, Segment inequality, Poincar\'e inequality, maximum principle, heat kernel estimates, Abresch-Gromoll estimate, and Anderson's bound on harmonic radius.
While many of the extensions are expected to be true and the proofs are analogous, there are notable exceptions. One of them is the bound on harmonic radius in the spirit of Anderson \cite{And}. In that paper, Anderson proved the following result.  Under suitable conditions on volume of balls or injectivity radius, if also the Ricci curvature is bounded, then $C^{1, \a}$ harmonic radius has a positive lower bound and the metric is $C^{1,\a} \cap W^{2, q}$ within such a radius. The lower bound of harmonic radius is very useful in many situations such that in establishing compactness of families of manifolds e.g.
However, one can not expect such a result under Bakry-\'Emery Ricci curvature bound.  Instead one can only expect $C^{\a} \cap W^{1, q}$ property for harmonic radius and the metric. To see this, let us recall the equation connecting metric $g$ and Ricci curvature under a harmonic coordinate chart:
\be
\lab{eqric-metric}
g^{ab}\frac{\pa^2 g_{kl}}{\pa v_a\pa v_b} + Q(\pa g,g)=-2(R_{kl}+\d_k\d_l L) + 2\d_k\d_l L.
\ee Here $Q$ is an expression involving quadratic quantity of $\pa g$.  Assuming  the Bakry-\'Emery Ricci  curvature is bounded, then the right hand side of the equation is the sum of an $L^\infty$ function and the Hessian of the function $L$. So if one wishes $g$ is a $W^{2, p}$ function, one needs to assume that the Hessian of $L$ is $L^p$. However this is not available for us.

The first result of this paper is a lower bound for such harmonic radius under suitable conditions on volume of balls.

In order to state the result rigorously, let us define the $W^{1,q}$ harmonic radius. Let $(\M^n,g)$ be an $n$-dimensional Riemannian manifold, and denote by $B_r(x)$ the geodesic ball in $\M$ centered at $x$ with radius $r$.

\begin{definition}
For $x\in \M$, the $W^{1,q}$ harmonic radius $r_h(x)$ at $x$ is the largest $r\geq 0$ such that there is a coordinate chart $\Phi=(v_1,v_2,\cdots,v_n):B_r(x)\rightarrow \mb{R}^n$ centered at $x$ such that $\Phi$ is a diffeomorphism onto its image, and\\
(1) $\Delta_g v_k=0$, $1\leq k\leq n$;\\
(2) let $g_{ij}=g(\pa_{v_i},\pa_{v_j})$ be the component of the metric $g$ considered as a function on $B_r(x)$. We have
\be\label{W1q bound of g}
\|g_{ij}-Id_{ij}\|_{C^{0}(B_r(x))}+r^{1-\frac{n}{q}}||\partial_{v_k} g_{ij}||_{L^q(B_r(x))}\leq \frac{1}{10},
\ee
where $Id_{ij}$ is the standard Euclidean metric on $\mathbb{R}^n$.
\end{definition}

Our first main result is

\begin{theorem}\label{thm harmonic radius bound}
Let $(\M^n,g)$ be a Riemannian n-manifold and $p$ be a point in $\M^n$. For each $q>2n$, there exist positive constants $\delta=\delta(n,q)$ and  $\theta=\theta(n,q)$ with the following properties.

(a) If $|Ric+\d^2 L|\leq n-1$ with $|\d L|\leq 1$, and
\be\label{volume condition1}
\vol(B_{\delta}(p))\geq (1-\delta)\vol(B_{\delta}(0^n)),
\ee
where $0^n$ denotes the origin of $\mb{R}^n$, then the $W^{1,q}$ harmonic radius $r_h(x)$ satisfies
\[
r_h(x)\geq \theta d(x,\pa B_{\delta^2}(p)),
\]
for all $x\in B_{\delta^2}(p)$.

(b) If $|Ric+\d^2 L|\leq n-1$ with $|\d L|\leq 1$, and the injectivity radius satisfies
\[
inj(x)\geq i_0>0
\]
in $B_{10}(p)$, then the $W^{1,q}$ harmonic radius $r_h(x)$ satisfies
\[
r_h(x)\geq \theta d(x,\pa B_{1}(p)),
\]
for all $x\in B_{1}(p)$.
\end{theorem}

\begin{remark}  Under the condition of the theorem, since $q>2n>n$, one knows that $W^{1, q}$ space embeds into $C^\a$ for $\a=1- \frac{n}{q}$. So we know that the metric is $C^\a$ automatically.
\end{remark}



\begin{remark}
Also indicated in the proof of Theorem \ref{thm harmonic radius bound} is the continuity of the $W^{1,q}$ harmonic radius.
\end{remark}

The next theorem of the paper is

\begin{theorem}\label{codimension 4}
Suppose a sequence of pointed manifolds $(\M^n_j, d_j, p_j)$ satisfies that
\[
|Ric_{M_j}+\d^2 L_j|\leq (n-1),\ with\ |\d L_j|\leq 1,
\]
and
\[
\vol(B_{10}(x))\geq \rho,\ \forall x\in\M_j,
\]
where $L_j\in C^{\infty}(\M_j)$, and $\rho>0$ is a constant.

If $(\M_j, d_j, p_j)\xrightarrow{d_{GH}}(X,d,p)$, then the singular set $\mathcal{S}$ satisfies
\[
dim(\mathcal{S})\leq n-4.
\]
\end{theorem}

\begin{remark} The constants $n-1$ and $1$ in the assumptions on Bakry-\'Emery Ricci curvature in the above theorems are chosen for convenience. They can be replaced by any positive constants.
\end{remark}



The rest of the paper is organized as follows. In section 2, we prove Theorem \ref{thm harmonic radius bound}. The proof follows the strategy in \cite{And} where a method of contradiction is used following a blow up procedure. Since our Ricci condition is weaker, a deeper analysis of the metric equation within harmonic radius is needed. These include mixed second derivative bound of Greens function and a careful covering argument.
The main issue is to prove $W^{1, q}$ convergence of the metrics in a blow up process.  One technical difficulty is that bounded sets in $W^{1, q}$ may not be compact in $W^{1, q'}$ for $q'<q$, which is different from the fact that bounded sets in $C^\a$ is compact in $C^{\a'}$ if $\a'<\a$.  An example is the sequence $f_k = \frac{1}{k} \sin (kx),  x \in [0, 2 \pi]$ in $W^{1, 2}([0, 2\pi])$. During the blow up process, it is easy to prove $C^\a_{loc}$ convergence of the metrics. However, $C^\a_{loc}$ convergence does not imply $W^{1, q}$ convergence. So we can not immediately deduce that the non-linear term $Q$ in (\ref{eqric-metric}) converges. In the classical case, one can prove $C^{1, \a}_{loc}$ convergence quickly and this already implies the convergence of the nonlinear term.

Theorem \ref{codimension 4} will be proved in Section 3. The proof is based on the techniques in \cite{ChNa2} and \cite{WZ}.
  A new ingradient is to show that the diagonal entries of the matrices in the Transformation Theorem are bounded away from 0. Some  other short cuts to the proof are also found. Together these seem to simplify the proof of the Transformation Theorem in \cite{ChNa2}, even in the original case.

\section{Bounds on harmonic radius and $\e$-regularity}

Let us start with a simple observation.
Recall the condition that
\be\label{basic assumption}
|Ric+\d^2L|\leq (n-1),\ |\d L|\leq 1.
\ee
 The theorem and proof are local in space. After  blowing up of metrics, this condition on Ricci curvature is always satisfied and actually becomes better.

Let $G(x,y)$ be the Green's function on $\M$. It is standard (using gradient bound on heat kernel etc) to show that
\be\label{Green's function estimate1}
|G(x,y)|\leq \frac{C}{d(x,y)^{n-2}},\ and\ |\d_yG(x,y)|\leq \frac{C}{d(x,y)^{n-1}}, \qquad d(x, y) \le 100.
\ee
Here and for the rest of this section, we use $C$ to denote constants depending only on the dimension $n$ and the parameters in the assumptions.

Suppose that $\Phi:U\rightarrow \mathbb{R}^n$ is a local coordinate chart on some open subset $U$ of $\M$. Denote by $\pa_{y_j}G(x,y)$ the jth component of $\d_y G(x,y)$. Then it is a harmonic function off the diagonal as a function of $x$. Thus, by the gradient estimate under Bakry-\'Emery Ricci condition, it follows that
\begin{lemma}\label{mixted hessian estimate G}
Under assumption \eqref{basic assumption}, it holds
\be\label{Green's function estimate2}
|\d_{x}\pa_{y_j}G(x,y)|\leq \frac{C}{d(x,y)^n}, \qquad \text{if} \quad d(x, y) \le 100, \quad B(y, 100) \subset U;
\ee
where $\d_{x}\pa_{y_j}G(x,y)$ is the gradient of $\pa_{y_j}G(x,y)$ as a function of $x$.
\end{lemma}
Here, gradient estimate works for \eqref{Green's function estimate2} because only mixed derivative is involved in the proof, which only requires the control of the quantities in \eqref{basic assumption} but not the whole curvature tensor.

  As a consequence of the Green's function estimates \eqref{Green's function estimate1} and \eqref{Green's function estimate2}, one can show

\begin{lemma}\label{lemma C alpha Green's function}
Assume that \eqref{basic assumption} holds. Then for any $r\leq 1$, $0<\a \le 1$, and $y, x_1,x_2\in B_{2r}(p)$ we have
\be\label{C alpha Green's function}
\al
|G(x_1,y)-G(x_2,y)|&\leq \frac{Cd(x_1,x_2)^{\a}}{\min(d(x_1,y)^{n-2+\a}, d(x_2,y)^{n-2+\a})};\\
|\pa_{y_j} G(x_1,y)-\pa_{y_j}G(x_2,y)| &\leq \frac{Cd(x_1,x_2)^{\a}}{\min(d(x_1,y)^{n-1+\a}, d(x_2,y)^{n-1+\a})}
\eal
\ee if $B(y, 100) \subset U$.
\end{lemma}

\proof We only prove the second estimate in \eqref{C alpha Green's function}. The proof of the first one is similar but easier.

If $d(x_1,y)\leq 2d(x_1,x_2)$, then \eqref{Green's function estimate1} implies that
\[
|\pa_{y_j} G(x_1,y)|\leq \frac{C}{d(x_1,y)^{n-1}}\leq \frac{Cd(x_1,y)^{\a}}{d(x_1,y)^{n-1+\a}}\leq \frac{Cd(x_1,x_2)^{\a}}{d(x_1,y)^{n-1+\a}},
\]
and
\[
\al
|\pa_{y_j} G(x_2,y)|\leq \frac{Cd(x_2,y)^{\a}}{d(x_2,y)^{n-1+\a}}\leq \frac{C[d(x_2,x_1)+d(x_1,y)]^{\a}}{d(x_2,y)^{n-1+\a}}\leq \frac{Cd(x_1,x_2)^{\a}}{d(x_2,y)^{n-1+\a}}.
\eal
\]
The estimates are similar when $d(x_2,y)\leq 2d(x_1,x_2)$.

Finally, if $\min(d(x_1,y),d(x_2,y))>2d(x_1,x_2)$, then by \eqref{Green's function estimate2}, one gets
\[
|\pa_{y_j} G(x_1,y)-\pa_{y_j}G(x_2,y)|\leq |\d_x\pa_{y_j} G|(x^*,y)d(x_1,x_2)\leq
\frac{Cd(x_1,x_2)}{d(x^*,y)^{n}}.
\]
Notice that in this case
\[
d(x^*,y)\geq d(x_i,y)-d(x^*,x_i)\geq d(x_i,y)-d(x_1,x_2) \geq \frac{1}{2}d(x_i,y)\geq d(x_1,x_2).
\]
Thus,
\[
|\pa_{y_j} G(x_1,y)-\pa_{y_j}G(x_2,y)|\leq \frac{Cd(x_1,x_2)^{\a}}{\min(d(x_1,y), d(x_2,y))^{n-1+\a}}.
\]
\qed

\proof[Proof of Theorem \ref{thm harmonic radius bound}:]

{\it Proof of $(a)$:} We will use the blow up argument in \cite{And} together with an extensive use of the "intrinsic" Green's function on the manifold $\M$.
Let us remark here that alternatively, one may also use the "extrinsic" Green's function, namely the Green's function of the operator $g^{ab} \frac{\pa^2}{\pa v_a \pa v_b}$ in the Euclidean space, after extending $g^{ab}$ suitably to the whole space.

Notice that by rescaling the metric $g$ by a factor $\delta^{-4}$, it amounts to prove the following statement. If $|Ric+\d^2 L|\leq (n-1)\delta^4$ with $|\d L|\leq \delta^2$, and
\be\label{volume condition2}
\vol(B_{\delta^{-1}}(p))\geq (1-\delta)\vol(B_{\delta^{-1}}(0^n)),
\ee
then the $W^{1,q}$ harmonic radius $r_h(x)$ satisfies
\[
r_h(x)\geq \theta d(x,\pa B_{1}(p)),
\]
for all $x\in B_{1}(p)$.

Under condition \eqref{volume condition2}, by Theorem 1.2 in \cite{WW} (see also Corollary 2.5 in \cite{ZZ}), we have for any $x\in B_{1}(p)$,
\[
\vol(B_{\delta^{-1}}(x))\geq \vol(B_{\delta^{-1}-1}(p))\geq e^{-2n\delta}(1-\delta)\vol(B_{\delta^{-1}-1}(0^n))= e^{-2n\delta}(1-\delta)^{n+1}\vol(B_{\delta^{-1}}(0^n)).
\]
Thus, when $\delta$ is small, it implies that
\be\label{volume nbhd}
\vol(B_{\delta^{-1}}(x))\geq [1-(3n+1)\delta]\vol(B_{\delta^{-1}}(0^n)).
\ee
Then applying the volume comparison theorem (Theorem 1.2 in \cite{WW}) one more time yields
\be\label{volume1}
\frac{\vol(B_{1}(x))}{\vol(B_{1}(0^n))}\geq e^{-2n\delta}\frac{\vol(B_{\delta^{-1}}(x))}{\vol(B_{\delta^{-1}}(0^n))}\geq [1-(3n+1)\delta]e^{-2n\delta}.
\ee
Again, if $\delta$ is small enough, one gets
\be\label{volume2}
\vol(B_{1}(x))\geq [1-(5n+1)\delta]\vol(B_{1}(0^n)).
\ee


Now we argue by contradiction to show that the theorem holds for some small $\delta$. Suppose that the theorem is not true. Then for any $\delta_j\rightarrow 0$, there is a sequence of manifolds $(\M_j, g_j)$, points $p_j\in\M_j$, and smooth functions $L_j$ such that
\[
|Ric_{\M_j}+\d^2L_j|_{g_j}\leq (n-1)\delta_j^{4},\quad |\d L_j|_{g_j}\leq \delta_j^2,
\]
and
\[
\vol(B_{\delta_j^{-1}}(p_j))\geq (1-\delta_j)\vol(B_{\delta_j^{-1}}(0^n)),
\]
but for some $z_j\in B_{1}(p_j)$, we have
\be\label{rhzj}
\frac{r_h(z_j)}{d(z_j,\pa B_{1}(p_j))}\rightarrow 0.
\ee
Without loss of generality, we may assume that $z_j$ is chosen so that the ratio $\frac{r_h(z)}{ d(z,\pa B_{1}(p_j))}$ reaches the minimum in $\overline{B_{1}(p_j)}$. It them implies that in the ball $B_{\frac{1}{2}d(z_j,\pa B_{1}(p_j))}(z_j)$, we have
\be\label{harmonic radius comparison}
r_h(z)\geq \frac{1}{2}r_h(z_j).
\ee

In the following, we  finish the proof in 5 steps.\\

\hspace{-.5cm}{\it Step 1: Blow-up and $C^{\a}$ convergence}

Denote by $r_j=r_h(z_j)$. Note that \eqref{rhzj} implies that $r_j\rightarrow 0$. Let us rescale the metric $g_j$ by the factor $r_j^{-2}$, i.e., $g_j\rightarrow r_j^{-2}g_j$. In the following, unless otherwise specified, all norms are taken with respect to the rescaled metric $r_j^{-2}g_j$. Hence, the manifold $(\M_j, g_j)$ satisfies
\be\label{rescaled BE Ricci}
|Ric_{\M_j}+\d^2L_j|\leq (n-1)r_j^2\delta_j^4,\quad |\d L_j|\leq r_j\delta_j^2,
\ee
and
\be\label{rescaled noncollapsing}
\vol(B_{(\delta_j r_j)^{-1}}(p_j))\geq (1-\delta_j)\vol(B_{(\delta_j r_j)^{-1}}(0^n)).
\ee
Also, from \eqref{rhzj}, one has
\be
d(z_j,\pa B_{r_j^{-1}}(p_j))\rightarrow\infty.
\ee
Gromov's precompact theorem implies that by passing to a subsequence, we have

\hspace{-.5cm}$(B_{d(z_j,\pa B_{r_j^{-1}}(p_j))}(z_j), d_j, z_j) \xrightarrow{d_{GH}}(\M_{\infty}, d_{\infty}, z_{\infty})$, where $d_j$ is the distance function related to the Riemannian metric $g_j$.
Then Corollary 4.8 and Remark 4.9 in \cite{WZ} and \eqref{volume2} conclude that $(\M_{\infty}, d_{\infty})=(\mathbb{R}^n, |\cdot|)$, where $|\cdot|$ denotes the standard Euclidean distance. Without loss of generality, we may assume that $z_{\infty}=0^n$.

On the other hand, by \eqref{harmonic radius comparison}, there is an open cover $\{B_{1/2}(z_{j_k})\}$ of $B_{\frac{1}{2}d(z_j,\pa B_{r_j^{-1}}(p_j))}(z_j)$ such that $B_{1/4}(z_{j_k})$ are mutually disjoint and there is a $W^{1,q}$ harmonic coordinate chart on all the balls. Since $q>2n$, by Sobolev embedding and the virtue of Lemma 2.1 in \cite{And} (See also \cite{Pe}), it actually holds that $(B_{d(z_j,\pa B_{r_j^{-1}}(p_j))}(z_j), g_j, z_j)\xrightarrow{C^{\a'}}(\mathbb{R}^n, Id_{ij}, 0^n)$ for any $\a'<1-\frac{n}{q}$ in Cheeger-Gromov sense. Moreover, we may assume that the index set $\{k\}$ is the same for all $B_{d(z_j,\pa B_{r_j^{-1}}(p_j))}(z_j)$, and $B_{1/2}(z_{j_k})\xrightarrow{d_{GH}}B_{1/2}(z_{\infty,k})$ for fixed $k$, as $j\rightarrow\infty$.

Next, we want to show that the convergence actually takes place in $C^{\a}$ and $W^{1,s}$ topology for any $\a\in(0,1)$ and $1<s<\infty$.

The estimates below work for all $\M_j$'s, so the subscript $j$ is dropped for convenience. In the remaining context of this step, denote by $\pa_{v_a}$ the partial derivative operator $\frac{\pa}{\pa v_a}$. In harmonic coordinates, the components of the Ricci curvature tensor can be expressed as
\[
-2R_{kl}=g^{ab}\pa_{v_a}\pa_{v_b}g_{kl} + Q(\pa g,g),
\]
where $Q(\pa g, g)$ is a quantity quadratic in the components of $\pa g$. The above equation may be viewed as a semi-linear elliptic equation of $g_{kl}$, namely,
\be\label{PDE g}
g^{ab}\pa_{ v_a}\pa_{v_b} g_{kl} + Q(\pa g,g)=-2(R_{kl}+\d_k\d_l L) + 2\d_k\d_l L.
\ee
Since $|\pa g|\in L^q$, we have $Q(\pa g,g)$ is uniformly bounded in $L^{\frac{q}{2}}$ with $\frac{q}{2}>n$. In addition, we know that $g^{ab}\in C^{\a'}$, $|Ric+\d^2 L|\in L^{\infty}$ and $|\d L|\in L^{\infty}$. Hence, the $W^{1,s}$ norm of $g_{kl}$ is uniformly bounded for all $1<s<\infty$, and by Sobolev embedding and the Arzela-Ascoli lemma, we see that $(B_{d(z_j,\pa B_{r_j^{-1}}(p_j))}(z_j), g_{j}, z_j)\xrightarrow{C^\a} (\mathrm{R}^n, Id, 0^n)$ for any $\a\in(0,1)$.  Here $g_j=(g_{kl})_j$ where $j$ is the index in the sequence of metrics.\\

\hspace{-.5cm}{\it Step 2: control of the $W^{1,s}$ norm on small scales}

To show the $W^{1,s}$ convergence of $g_j$, for any $z\in B_{1/2}(z_{j_k})$, let $\eta>0$ be an arbitrary constant such that $B_{2\eta}(z)\subseteq B_{1/2}(z_{j_k})$. Choose a cut-off function $\phi$ supported in $B_{2\eta}(z)$ such that $\phi=1$ in $B_{3\eta/2}(z)$ and $|\Delta \phi|+|\d \phi|^2\leq C/\eta^2$. For the existence, see e.g. Lemma 1.5 in \cite{WZ}.
Also, for simplicity of presentation we temporarily drop the index $j$ in the metrics, unless there is confusion.  Then for
\be
\lab{defhkl}
h_{kl}= g_{kl}-g_{kl}(z),
\ee from \eqref{PDE g}, we have
\[
\al
-\frac{1}{2}\Delta (\phi h_{kl})
&= \phi\left(R_{kl}+\d_k\d_l L\right) - \phi\d_k\d_l L-\phi Q(\pa g, g) + 2g^{ab}\pa_{v_a} h_{kl}\pa_{v_b}\phi +  g^{ab}h_{kl}\pa_{v_a}\pa_{v_b}\phi\\
&: = I_1+I_2+I_3+I_4+I_5.
\eal
\]

It follows from Green's formula that
\[
\phi h_{kl}(x)= 2\int_{\M} G(x,y)(I_1+\cdots+I_5)(y)dy,
\]
and hence
\be\label{Green's formula partial g}
\pa_{v_m} g_{kl}(x)=\pa_{v_m} (\phi h_{kl})(x) = 2\int_{\M} \pa_{v_m(x)} G(x,y)(I_1+\cdots+I_5)(y)dy
\ee
in $B_{\eta}(z)$.

For any $q<s<\infty$, let $s'=\frac{s}{s-1}$ and $\psi\in C_0^{\infty}(B_{\eta}(z))$. In the following the $L^s$ and $L^t$ norms are taken over $B_{2\eta}(z)$.
From \eqref{Green's formula partial g}, one has
\be\label{integral Green's formula partial g}
\int_{B_{\eta}(z)} \pa_{v_m} g_{kl}(x)\psi(x) dx = 2\int_{B_{\eta}(z)}\left(\int_{\M} \pa_{v_m(x)}G(x,y)(I_1+\cdots+I_5)(y)dy\right)\psi(x)dx.
\ee
Firstly, by \eqref{Green's function estimate1} and \eqref{rescaled BE Ricci}, we have
\be\label{term1}
\al
2\int_{B_{\eta}(z)}\left(\int_{\M} \pa_{v_m(x)}G(x,y)I_1dy\right)\psi(x)dx\leq & Cr_j^2\int_{B_{\eta}(z)}\left(\int_{B_{2\eta}(z)}\frac{1}{d^{n-1}(x,y)}dy\right)\psi(x)dx\\
\leq & C\eta r_j^2 \vol(B_{\eta}(z))^{1/s}\|\psi\|_{L^{s'}(B_{\eta}(z))}.
\eal
\ee

Next, by writing $h_{kl}=(g_{kl}-Id_{kl})-(g_{kl}-Id_{kl})(z)$ and using the $C^{\a}$ boundedness of $|g-Id|$, we have for any $\a\in (1-\frac{n}{s}, 1)$ that
\be\label{term2}
\al
&2\int_{B_{\eta}(z)}\left(\int_{\M} \pa_{v_m(x)}G(x,y)I_5dy\right)\psi(x)dx\\
\leq & \frac{C}{\eta^2}\|g-Id\|_{C^{\a}}\int_{B_{\eta}(z)}\left(\int_{B_{2\eta}(z)}d(y,z)^{\a}\frac{1}{d(x,y)^{n-1}}dy\right)\psi(x)dx\\
\leq & \frac{C\|g-Id\|_{C^{\a}}}{\eta^{1-\a}}\vol(B_{\eta}(z))^{1/s}\|\psi\|_{L^{s'}(B_{\eta}(z))}\\
\leq & C\|g-Id\|_{C^{\a}}\vol(B_{\eta}(z))^{\frac{1}{s}-\frac{1-\a}{n}}\|\psi\|_{L^{s'}(B_{\eta}(z))}.
\eal
\ee

For $I_4$, using integration by parts yields
\[
\al
&2\int_{B_{\eta}(z)}\left(\int_{\M}\pa_{v_m(x)}G(x,y)I_4dy\right)\psi(x)dy dx\\
=&\underbrace{\int_{B_{\eta}(z)} \int_{\M} \pa_{v_m(x)}\pa_{v_a(y)}G(x,y)g^{ab}h_{kl}\pa_{v_b} \phi\,\psi(x) dydx}_{(1)} + \underbrace{\int_{B_{\eta}(z)} \int_{\M} \pa_{v_m(x)}G(x,y)\pa_{v_a} g^{ab}h_{kl}\pa_{v_b} \phi\, \psi(x) dydx}_{(2)} \\
&\qquad\qquad\qquad + \underbrace{\int_{B_{\eta}(z)} \int_{\M} \pa_{v_m(x)}G(x,y)g^{ab}h_{kl}\pa_{v_a}\pa_{v_b}\phi\, \psi(x) dy dx}_{(3)}\\
:=&(1)+(2)+(3).
\eal
\]

From \eqref{term2}, we have
\be\label{(3)}
(3)\leq C\|g-Id\|_{C^{\a}}\vol(B_{\eta}(z))^{\frac{1}{s}-\frac{1-\a}{n}}\|\psi\|_{L^{s'}(B_{\eta}(z))}.
\ee
Similarly, by Lemma \ref{mixted hessian estimate G}, we get
\be\label{(1)}
\al
(1)\leq& \frac{C\|g-Id\|_{C^{\a}}}{\eta}\int_{B_{\eta}(z)}\bigg(\int_{B_{2\eta}(z)\setminus B_{3\eta/2}(z)}\frac{d(y,z)^{\a}}{d(x,y)^n}dy\bigg)\psi(x)dx\\
\leq & C\|g-Id\|_{C^{\a}}\eta^{\a-1}\int_{B_{\eta}(z)}\psi(x)dx\\
\leq & C\|g-Id\|_{C^{\a}}\eta^{\a-1+\frac{n}{s}}\|\psi\|_{L^{s'}(B_{\eta}(z))}\\
\leq &C\|g-Id\|_{C^{\a}}\vol(B_{\eta}(z))^{\frac{1}{s}-\frac{1-\a}{n}}\|\psi\|_{L^{s'}(B_{\eta}(z))}.
\eal
\ee
Since $g^{ab}\in W^{1,q}$, it follows that
\be\label{(2)}
\al
(2)\leq & \frac{C\|g-Id\|_{C^{\a}}}{\eta}\int_{B_{\eta}(z)}\bigg(\int_{B_{2\eta}(z)\setminus B_{3\eta/2}(z)}\left|\pa_{v_a} g^{ab}\right|\frac{d(y,z)^{\a}}{d(x,y)^{n-1}}dy\bigg)\psi(x)dx\\
\leq& C\|g-Id\|_{C^{\a}}\eta^{\a-n}\int_{B_{\eta}(z)}\bigg(\int_{B_{2\eta}(z)}\left|\pa_{v_a} g^{ab}\right|dy\bigg)\psi(x)dx\\
\leq& C\|\pa g\|_{L^q}\|g-Id\|_{C^{\a}} \eta^{\a-n+\frac{n(q-1)}{q}}\int_{B_{\eta}(z)}\psi(x)dx\\
\leq& C\|g-Id\|_{C^{\a}} \eta^{\a-\frac{n}{q}+\frac{n}{s}}\|\psi\|_{L^{s'}(B_{\eta}(z))}\\
\leq& C\|g-Id\|_{C^{\a}}\vol(B_{\eta}(z))^{\frac{1}{s}-\frac{1-\a}{n}}\|\psi\|_{L^{s'}(B_{\eta}(z))}.
\eal
\ee Here we have used $q>2n$.
Thus, putting \eqref{(3)}, \eqref{(1)}, and \eqref{(2)} together, one has
\be\label{term3}
2\int_{B_{\eta}(z)}\left(\int_{\M}\pa_{v_m(x)}G(x,y)I_4dy\right)\psi(x)dx\leq C\|g-Id\|_{C^{\a}}\vol(B_{\eta}(z))^{\frac{1}{s}-\frac{1-\a}{n}}\|\psi\|_{L^{s'}(B_{\eta}(z))}.
\ee

Moreover, since $Q(\pa g,g)\in L^{q/2}$, applying H\"older inequality followed by Young's inequality imply that
\[
\al
&2\int_{B_{\eta}(z)}\left(\int_{\M} \pa_{v_m(x)}G(x,y)I_3dy\right)\psi(x)dx\\
\leq & \|Q(\pa g,g)\|_{L^{q/2}}\left[\int_{B_{2\eta}(z)}\left(\int_{B_{\eta}(z)}|\pa_{v_m(x)}G(x,y)|\psi(x)dx\right)^{\frac{q}{q-2}}dy\right]^{\frac{q-2}{q}}\\
\leq & C\|\pa g\|^2_{L^{q}} \sup_{y\in B_{2\eta}(z)}\|\pa_{v_m}G(\cdot,y)\|_{L^t}\|\psi\|_{L^{s'}(B_{\eta}(z))},
\eal
\]
where $\|\pa g\|_{L^s}$ means the sum of the $L^s$ norms of all the components of $\pa g$, and $t$ satisfies
\[
1+\frac{q-2}{q}=\frac{1}{t}+\frac{1}{s'},\ i.e.,\ t=\frac{sq}{s(q-2)+q}.
\]
Noticing that $q>2n$, it is easy to check that $(n-1)t<n$, and hence
\[
\|\pa_{v_m}G(\cdot,y)\|_{L^t(B_{2\eta}(z))}\leq C\left(\int_{B_{2\eta}(z)}\frac{1}{d(x,y)^{(n-1)t}}dx\right)^{1/t}\leq C\eta^{n/t-(n-1)}.
\]
Thus, we have
\be\label{term4}
\al
2\int_{B_{\eta}(z)}\left(\int_{\M} \pa_{v_m(x)}G(x,y)I_3dy\right)\psi(x)dx\leq&
C\eta^{\frac{n}{t}-(n-1)+\frac{n(s-q)}{sq}}\|\pa g\|_{L^s}\|\psi\|_{L^{s'}(B_{\eta}(z))}.
\eal
\ee

Finally,
\[
\al
&2\int_{B_{\eta}(z)}\left(\int_{\M} \pa_{v_m(x)}G(x,y)I_2 dy\right)\psi(x)dx\\
= & 2\int_{B_{\eta}(z)}\left(\int_{\M} \pa_{v_m(x)}G(x,y)\phi(y)[\pa_{v_k(y)}\pa_{v_l(y)}L - \Gamma_{kl}^n\pa_{v_n(y)}L] dy\right)\psi(x)dx\\
=& -2\int_{B_{\eta}(z)}\bigg(\int_{\M} \left[\pa_{v_m(x)}\pa_{v_k(y)}G(x,y)\phi(y)+\pa_{v_m(x)}G(x,y)\pa_{v_k(y)}\phi(y)\right]\pa_{v_l(y)}L\\
&\hspace{2cm} + \pa_{v_m(x)}G(x,y)\phi(y)\Gamma_{kl}^n\pa_{v_n(y)}L dy\bigg)\psi(x)dx
\eal
\]
For the second and third terms above, since $\Gamma_{kl}^n\in L^q$, $|\d \phi|\leq C/\eta$ and $|\d L|\leq Cr_j$, as in \eqref{term1} we get
\be\label{1st and 2nd}
\al
& -2\int_{B_{\eta}(z)}\left(\int_{\M} \pa_{v_m(x)}G(x,y)\pa_{v_k(y)}\phi(y)\pa_{v_l(y)}L+\pa_{v_m(x)}G(x,y)\phi(y)\Gamma_{kl}^n\pa_{v_n(y)}L dy\right)\psi(x)dx\\
& \leq  Cr_j\vol(B_{\eta}(z))^{1/s}\|\psi\|_{L^{s'}(B_{\eta}(z))}.
\eal
\ee
For the first term, using H\"older inequality gives
\be\label{3rd}
\al
&-2\int_{B_{\eta}(z)}\left(\int_{\M} \pa_{v_m(x)}\pa_{v_k(y)}G(x,y)\phi(y)\pa_{v_l(y)}L dy\right)\psi(x)dx\\
\leq & 2\left[\int_{B_{2\eta}(z)}\left|\int_{B_{\eta}(z)}\pa_{v_m(x)}\pa_{v_k(y)}G(x,y)\psi(x)dx\right|^{s'}dy\right]^{1/s'}\left(\int_{B_{2\eta}(z)}|\d L|^sdy\right)^{1/s}\\
\leq& Cr_j\vol(B_{\eta}(z))^{1/s}\|\psi\|_{L^{s'}(B_{\eta}(z))},
\eal
\ee
where in the last step we have used the fact that
\[
\left[\int_{B_{2\eta}(z)}\left|\int_{B_{\eta}(z)}\pa_{v_m(x)}\pa_{v_k(y)}G(x,y)\psi(x)dx\right|^{s'}dy\right]^{1/s'}\leq C\|\psi\|_{L^{s'}(B_{\eta}(z))}.
\]
This is because $G(x,y)$ is the kernel of the Laplace operator $g^{ab}\frac{\pa^2}{\pa v_a \pa v_b}$, and then we may rewrite equation $g^{ab}\frac{\pa^2 u}{\pa v_a \pa v_b}=\psi$ as $g^{ab}(z)\frac{\pa^2 u}{\pa v_a \pa v_b}= (g^{ab}(z)-g^{ab})\frac{\pa^2 u}{\pa v_a \pa v_b}+\psi$ and use the fact that $g^{ab}\in C^{\a}$ and $\pa_x\pa_y G_z(x,y)$ is a Caldr\'on-Zygmund kernel. Here $G_z(x,y)$ denotes the kernel of the operator $g^{ab}(z)\frac{\pa^2 }{\pa v_a \pa v_b}$.

Thus, combining \eqref{1st and 2nd} and \eqref{3rd}, we get
\be\label{term5}
\al
2\int_{B_{\eta}(z)}\left(\int_{\M} \pa_{v_m(x)}G(x,y)I_2 dy\right)\psi(x)dx\leq Cr_j\vol(B_{\eta}(z))^{1/s}\|\psi\|_{L^{s'}(B_{\eta}(z))}.
\eal
\ee

Putting \eqref{term1}, \eqref{term2}, \eqref{term3}, \eqref{term4} and \eqref{term5} together in \eqref{integral Green's formula partial g}, we obtain
\[
\al
&\int_{B_{\eta}(z)} \pa_{v_m} g_{kl}(x)\psi(x) dx\\
\leq& C\eta^a\|\pa g\|_{L^s}\|\psi\|_{L^{s'}(B_{\eta}(z))}+C(r_j+\|g-Id\|_{C^{\a}})\vol(B_{\eta}(z))^{\frac{1}{s}-\frac{1-\a}{n}}\|\psi\|_{L^{s'}(B_{\eta}(z))},
\eal
\]
where $0<a<1$ is a constant.

Therefore, by taking supremum over $\psi$ on the left hand side and summing up all indices, we have, after recalling that $g=g_j$ in the sequence of metrics, that
\be\label{Ls pa g}
\al
\|\pa g_j\|_{L^s(B_{\eta}(z))}\leq & C\eta^a\|\pa g_j\|_{L^s(B_{2\eta}(z))}+C(r_j+\|g_j-Id\|_{C^{\a}})\vol(B_{\eta}(z))^{\frac{1}{s}-\frac{1-\a}{n}}.
\eal
\ee

\hspace{-.5cm}{\it Step 3: covering argument and $W^{1,s}$ convergence}

Even though the second term on the right is approaching 0, estimate \eqref{Ls pa g} above cannot be applied directly to derive the $W^{1,s}$ convergence of the metrics due to the difference between the size of the balls centered at $z$. The idea is to make the sizes of the balls on both sides even. For this purpose, we take advantage of the Whitney covering, which allows us to cover a big ball with countable many small balls while none of the small balls will escape the big ball and the overlapping number can be uniformly controlled. Eventually, the $L^s$ norm of $\pa g_j$ on both sides will be on the same ball, and $\delta$ can be replaced by the largest diameter of the balls in the covering. Hence, by making $\delta$ small enough, one will get the $\|\pa g_j\|_{L^s}\rightarrow 0$ as desired.

We choose the Whitney covering $\mathcal{B}$ of $B_{1/2}(x_{j_k})$ as follows: for some $m_0$ chosen below, cover the ball $B_{\frac{1}{2}-\frac{1}{2^{m_0}}}(x_{j_k})$ with finitely many balls $B_{\frac{1}{2^{m_0+1}}}$ of fixed size, cover the sphere $\pa B_{\frac{1}{2}-\frac{1}{2^{m}}}(x_{j_k})$, $m\geq m_0$ with balls $B_{\frac{1}{2^{m+1}}}$ and cover the remaining region in the annulus $B_{\frac{1}{2}-\frac{1}{2^{m+1}}}(x_{j_k})\setminus B_{\frac{1}{2}-\frac{1}{2^{m}}}(x_{j_k})$ also by balls $B_{\frac{1}{2^{m+2}}}$, where $B_r$ denotes a ball with radius $r$. Hence, $B_{1/2}(x_{j_k})$ is the union of all the balls in the covering. In addition, we may require all the balls with half of the radius to be disjoint.

Denote the number of balls with the radius $\frac{1}{2^{m+1}}$ by $K_m$. To estimate $K_m$, first notice that $K_{m_0}$ is a constant only depending on $m_0$ and the parameters in the assumptions of the theorem. Then for each $m\geq m_0+1$, since the balls $B_{\frac{1}{2^{m+1}}}$ are contained in the annulus $B_{\frac{1}{2}-\frac{1}{2^{m+1}}}\setminus B_{\frac{1}{2}-\frac{1}{2^{m-2}}}$, and these balls with half of the radius $\frac{1}{2^{m+2}}$ are disjoint and volume noncollapsed, we have
\be\label{Km}
cK_{m}\left(\frac{1}{2^{m+2}}\right)^n\leq C\left[\left(\frac{1}{2}-\frac{1}{2^{m+1}}\right)^n-\left(\frac{1}{2}-\frac{1}{2^{m-2}}\right)^n\right],
\ee
which implies that
\be\label{Km bound}
K_m\leq C2^{m(n-1)}.
\ee
In the above, the right hand side of \eqref{Km} is derived by integrating the area of geodesic spheres between $B_{\frac{1}{2}-\frac{1}{2^{m+1}}}$ and $B_{\frac{1}{2}-\frac{1}{2^{m-2}}}$, and using the volume element comparison. See e.g. \cite{WW} (also \cite{WZ} or \cite{ZZ}).

In \eqref{Ls pa g}, we replace $B_{\eta}(z)$ on the left hand side by the balls in the Whitney cover $\mathcal{B}$ and sum up all the integrals. Note that balls $\{B_{2\eta}(z)\}$ also form a Whitney cover of $B_{1/2}(x_{j_k})$, denoted by $2\mathcal{B}$, and the overlapping number $N$ is uniformly bounded regardless the choice of $m_0$. By using \eqref{Km bound}, one has
\[
\al
&\|\pa g_j\|^s_{L^s(B_{1/2}(x_{j_k}))}\\
\leq& \sum_{B\in\mathcal{B}}\|\pa g_j\|^s_{L^s(B)}\\
\leq & C\frac{1}{2^{asm_0}}\sum_{2B\in 2\mathcal{B}}\|\pa g_j\|^s_{L^s(2B)}\\
& +C(r_j+\|g_j-Id\|_{C^{\a}})^s\left(K_{m_0}(\frac{1}{2^{m_0+1}})^{n-(1-\a)ns}+\sum_{m=m_0+1}^{\infty}K_m(\frac{1}{2^{m+1}})^{n-(1-\a)ns}\right)\\
\leq & \frac{C}{2^{asm_0}}N\|\pa g_j\|^s_{L^s(B_{1/2}(x_{j_k}))}\\
& +C(r_j+\|g_j-Id\|_{C^{\a}})^s\left(C(m_0)+C\sum_{m=m_0+1}^{\infty}(\frac{1}{2})^{m[1-(1-\a)ns]}\right).
\eal
\]
Therefore, one can see that for any $s>q$, by choosing $m_0$ and $\a$ so that $\frac{C}{2^{asm_0}}N<\frac{1}{2}$ and $1-(1-\a)ns>0$, it follows that
\[
\al
\|\pa g_j\|^{s}_{L^s(B_{1/2}(x_{j_k}))}\leq &
\frac{1}{2}\|\pa g_j\|^{s}_{L^s(B_{1/2}(x_{j_k}))} +  C(r_j+\|g_j-Id\|_{C^{\a}}),
\eal
\]
which amounts to
\be
\lab{dgto0ls}
\|\pa g_j\|_{L^s(B_{1/2}(x_{j_k}))}\leq 2C(r_j+\|g_j-Id\|_{C^{\a}})\rightarrow 0,
\ee
since both $r_j\rightarrow 0$ and $\|g_j-Id\|_{C^{\a}}\rightarrow 0.$

This implies the $W^{1,s}$ convergence of $(B_{d(z_j,\pa B_{r_j^{-1}}(p_j))}(z_j), g_j, z_j)$. Indeed, for a fixed small radius $\eta>0$ and any compact subset $D \subset B_{d(z_j,\pa B_{r_j^{-1}}(p_j))}(z_j)$, by using volume comparison, one can get a uniform control (independent of $j$) of the number of points in the $\eta/4$-net of $D$. When $\eta$ is chosen small enough, we can get a covering of $D$ with balls with radius $\eta/2$, such that on each ball there is a harmonic coordinate chart. Then by using a similar argument as in the proof of Whitney embedding theorem, we may construct an smooth embedding from $D$ to $\mb{R}^N$. Moreover, under this embedding, the local images are graphs. Since from \eqref{dgto0ls}, we have the $W^{1,s}$ convergence of the metrics in harmonic coordinates to the Euclidean metric on $\mb{R}^n$, the transition functions of the covering of $D$ are converging in $W^{2,s}$ to the transition functions of $\mb{R}^n$. Also, for the same reason, the local graphs are converging in $W^{2,s}$ norm. And hence, when $j$ is large enough there exist diffeomorphisms between exhausting compact sets in $\mb{R}^n$ and sets $B_{d(z_j,\pa B_{r_j^{-1}}(p_j))}(z_j)$ such that the pull back metrics of $g_j$ are converging in $W^{1,s}$ norm to $Id_{ij}$. See e.g. Proposition 12 in \cite{HH} for more details. One can also find similar arguments on $W^{1, s}$ convergence in \cite{AC}. Note in that paper, one assumes the Ricci curvature is bounded from below. However, this assumption is only used to deduce volume comparison results which also holds in our situation. So the proof is valid in our case.

Therefore, we have shown that $(B_{d(z_j,\pa B_{r_j^{-1}}(p_j))}(z_j), g_j, z_j)\xrightarrow{C^{\a}\cap W^{1,s}}(\mathbb{R}^n,Id_{ij}, 0^n)$. \\

\hspace{-.5cm}{\it Step 4: constructing harmonic coordinates on balls with radius larger than 1}

From step 3 and the definition of Cheeger-Gromov convergence, we have that for $j$ sufficiently large, there is a diffeomorphism $F_j: \mathbb{R}^n\rightarrow \M_j$ such that $F_j^{*}g_j$ converges to $Id$ in $W^{1,s}$ topology on compact subsets of $\mathbb{R}^n$. Thus, there is a covering of $B_2(0^n)$, denoted by $\{B_{i}\}$, with balls of radius $1/2$, on each of which there is a harmonic coordinate chart $\{v_1,\cdots, v_n\}$ uniformly bounded in $C^{1,\a}\cap W^{2,s}$. In fact, the Laplace equation in Euclidean coordinates reads
\[
\Delta_j v_k = \frac{1}{\sqrt{det(h_j)}}\frac{\pa }{\pa x_a}(\sqrt{det(h_j)}h^{ab}_j\frac{\pa v_k}{\pa x_b})=0.
\]Here $\Delta_j$ is the Laplace operator of the metric $h_j=F_j^*g_j$, and $\{x_k\}$ are the standard Euclidean coordinates.
Thus, the $W^{2,s}$ bound of $v_k$ follows from the $W^{1,s}$ bound of $h_j$ and standard elliptic regularity theory.

To construct larger harmonic coordinate chart with respect to $h_j$, let $y_k=y_k(j)$ be the solution of the Dirichlet problem
\[
\Delta_j y_k=0,\ in\ B_{3/2}(0^n);\quad y_k=x_k\ on\ \pa B_{3/2}(0^n).
\]

We first show that $\{y_k\}$ gives a harmonic coordinate chart on $B_{5/4}(0^n)$. Indeed, let $w_k=x_k-y_k$, then
\be\label{laplace w}
\Delta_j w_k=\Delta_j x_k,\ in\ B_{3/2}(0^n),\ and\ w_k=0\ on\ \pa B_{3/2}(0^n).
\ee
In Euclidean coordinates $\Delta_j x_k= \frac{1}{\sqrt{det(h_j)}}\frac{\pa }{\pa x_a}(\sqrt{det(h_j)}h^{ab}\frac{\pa x_k}{\pa x_b})$. Since the metrics $h_j$ converges in $W_{loc}^{1,s}$ norm to the Euclidean metric, it implies that
\be\label{Ls of Delta x}
\|\Delta_j x_k\|_{L^s(B_{3/2}(0^n))}\rightarrow 0.
\ee
Thus, by the maximal principle, one gets that
\be\label{Linfty of w}
\|w_k\|_{L^{\infty}(B_{3/2}(0^n))}\rightarrow 0.
\ee
It then follows from the gradient estimate under Bakry-\'Emery Ricci condition that
\be\label{Linfty of dw}
\|\d_j w_k\|_{L^{\infty}(B_{5/4}(0^n))}\rightarrow 0.
\ee
Let $\phi$ be a cut-off function supported in $B_{3/2}(0^n)$ such that $\phi=1$ in $B_{11/8}(0^n)$ and $|\Delta_j \phi|+|\d_j \phi|\leq C$. From \eqref{laplace w} and Green's formula, we have
\[
\phi w_k(x)=-\int_{\M} G(x,y)\left[\phi\Delta_j x_k + 2<\d_j \phi, \d_j w_k> + w_k\Delta_j\phi\right]~dy.
\]
From Lemma \ref{lemma C alpha Green's function}, we have for $x_1,x_2\in B_{5/4}(0^n)$ that
\[
\al
&|\d_j w_k(x_1)-\d_j w_k(x_2)|\\
=&|\d_j (\phi w_k)(x_1)-\d_j (\phi w_k)(x_2)|\\
\leq & \int_{\M} |\d_j G(x_1,y) - \d_j G(x_2,y)|\left|\phi\Delta_j x_k + 2<\d_j \phi, \d_j w_k> + w_k\Delta_j\phi\right| dy\\
\leq& \int_{B_{3/2}(0^n)} \frac{Cd_j(x_1,x_2)^{\a}}{d_j(x_1,y)^{n-1+\a}}(|\Delta_j x_k|+|\d_j w_k|+|w_k|)dy.
\eal
\]
Then by H\"older inequality, \eqref{Ls of Delta x}, \eqref{Linfty of w} and \eqref{Linfty of dw}, it implies that for $\a\in(0, 1-\frac{n}{s})$
\[
\|w_k\|_{C^{1,\a}(B_{5/4}(0^n))}\rightarrow 0.
\]
In particular, $\{y_1,\cdots, y_n\}$ forms a coordinate system on $B_{5/4}(0^n)$ when $j$ is big enough.\\

\hspace{-.5cm}{\it Step 5: larger $W^{1,q}$ harmonic radius and contradiction}

It is left to show that $\eqref{W1q bound of g}$ is satisfied under $\{y_k\}$ with $r=\frac{5}{4}$. For this, we need to show that $y_k$ converges in $W^{2,s}$ norm. In each $B_i$, under the harmonic coordinates $\{v_1,\cdots,v_n\}$, \eqref{laplace w} can be written as
\[
h_j^{mn}\frac{\pa^2 w_k}{\pa v_m \pa v_n}=\Delta_j x_k,
\]
For any point $v_0\in B_i$, let $\phi$ be a cut-off function supported in $B_{2\eta}(v_0)$ such that $\phi=1$ in $B_{\eta}(v_0)$ and $|\Delta\phi|+|\d \phi|^2\leq C/\eta^2$, where $\eta$ is a small constant which will be determined later.

Then, we have
\[
\al
h_j^{ab}(v_0)\frac{\pa^2 (\phi w_k)}{\pa v_a\pa v_b}=&\left(h_j^{ab}(v_0)-h_j^{ab}(v)\right)\frac{\pa^2 (\phi w_k)}{\pa v_a\pa v_b} + h_j^{ab}(v)\frac{\pa^2 (\phi w_k)}{\pa v_a\pa v_b}\\
=&\left(h_j^{ab}(v_0)-h_j^{ab}(v)\right)\frac{\pa^2 (\phi w_k)}{\pa v_a\pa v_b} + \phi\Delta_j x_k+ 2h_j^{ab}\frac{\pa \phi}{\pa v_a}\frac{\pa w_k}{\pa v_b} + w_kh_j^{mn}\frac{\pa^2 \phi}{\pa v_a\pa v_b}\\
: = & F(v).
\eal
\]
Since $h_j^{mn}(v_0)$ is a constant satisfying $(1-c)Id\leq h_j(v_0)\leq (1+c)Id$, it follows that $\pa_x\pa_y G_{v_0}(x,y)$ is a Calder\'on-Zygmund kernel, where $G_{v_0}(x,y)$ is the kernel of the operator $h_j^{ab}(v_0)\frac{\pa^2 }{\pa v_a\pa v_b}$. Hence, it defines a Calder\'on-Zygmund operator bounded on $L^s$ space, namely, we have
\[
\|\frac{\pa^2 (\phi w_k)}{\pa v_a \pa v_b}\|_{L^s(B_{2\eta}(v_0))}\leq C\|F(v)\|_{L^s(B_{2\eta}(v_0))}.
\]
By the $C^{\a}$ boundedness of $h_j^{ab}$, one derives
\[
\|F\|_{L^s(B_{2\eta}(v_0))}\leq C\eta^{\a}\|\frac{\pa^2 (\phi w_k)}{\pa v_a \pa v_b}\|_{L^s(B_{2\eta}(v_0))} + \frac{C}{\eta^2}\left[\|\Delta_j x_k\|_{L^s}+\|\d_j w_k\|_{L^{\infty}}+\|w\|_{L^{\infty}}\right].
\]
By choosing $\eta$ small enough, we can make $C\eta^{\a}<\frac{1}{2}$, and hence from \eqref{Ls of Delta x}, \eqref{Linfty of w}, \eqref{Linfty of dw}, it follows that
\[
\|\frac{\pa^2 w_k}{\pa v_a \pa v_a}\|_{L^s(B_{\eta}(v_0))}\leq\|\frac{\pa^2 (\phi w_k)}{\pa v_a \pa v_b}\|_{L^s(B_{2\eta}(v_0))}\leq C\left[\|\Delta_j x_k\|_{L^s}+\|\d_j w_k\|_{L^{\infty}}+\|w\|_{L^{\infty}}\right]\rightarrow 0.
\]
Through a standard covering argument, it is easy to see that
\[
\|w_k\|_{W^{2,s}(B_{5/4}(0^n))}\rightarrow 0.
\]
This is sufficient to indicate that
\[
\|\pa_{y_m}h_j(\pa_{y_k}, \pa_{y_l})\|_{L^q(B_{5/4}(0^n))}=\left\|\frac{\pa x_a}{\pa y_m}\pa_{x_a}\left[h(\pa_{x_c},\pa_{x_d})\frac{\pa x_c}{\pa y_k}\frac{\pa x_d}{\pa y_l}\right]\right\|_{L^q(B_{5/4}(0^n))}\rightarrow 0.
\]
Therefore, it follows that $\{y_1,\cdots, y_n\}$ is a $W^{1,q}$ harmonic coordinate chart on $B_{5/4}(0^n)$ when $j$ is large enough, which in term induces a $W^{1,q}$ harmonic coordinate chart on a ball centered at $z_j$ with radius larger than 1 in $\M_j$, and contradicts to the hypothesis that the $W^{1,q}$ harmonic radius $r_h(z_j)=1$. \\

{\it Proof of (b):} The proof of part (b) is similar. One just needs to first notice that by modifying the proof of part (a) slightly, one can derive the following compactness result for manifolds under $W^{1,s}$ convergence.

\begin{theorem}\label{thm compactness}
Let $(\M^n_j, g_j, p_j)$ be a sequence of pointed Riemannian manifolds satisfying that $|Ric_{\M_j}+\d^2 L_j|\rightarrow 0$, $|\d L_j|\rightarrow 0$, and $(\M^n_j, d_j, p_j)\xrightarrow{d_{GH}}(\M_{\infty}, d_{\infty}, p)$. Suppose also the $W^{1, s}$ harmonic radius is bounded from below by a uniform positive constant for all $s>2n$. Then there is a $C^{\a}\cap W^{1,s}$ Riemannian metric $g_{\infty}$ on $\M_{\infty}$ such that $(\M^n_j, g_j, p_j)\xrightarrow{C^{\a}\cap W^{1,s}}(\M_{\infty}, g_{\infty}, p)$ in Cheeger-Gromov sense for any $0<\a<1$ and $1<s<\infty$.
\end{theorem}

Indeed, from the assumption and Arzela-Ascoli Lemma, we immediately get $C^{\a'}$ convergence of the sequence of manifolds for any $0<\a'<1-\frac{n}{s}$. To show the $W^{1,s}$ convergence, we just need to replace the Euclidean metric $Id$ in step 2 and 3 in the proof of part (a) by $g_{\infty}$, and estimate $\|\pa g - \pa g_{\infty}\|_{L^s}$ instead of $\|\pa g\|_{L^s}$. So instead of (\ref{dgto0ls}), one obtains:
\[
\|\pa g_j-\pa g_{\infty} \|_{L^s(B_{1/2}(x_{j_k}))}\leq 2C(\e_j+\|g-g_\infty\|_{C^{\a}})\rightarrow 0.
\]Here we have assumed, without loss, the harmonic radii is bounded from below by $1$. Also
\[
\e_j= ||Ric_{\M_j}+\d^2 L_j||_\infty+ ||\d L_j||_\infty.
\] Now the convergence in $C^\a$ sense follow from Sobolev imbedding.\\

With Theorem \ref{thm compactness} in hand, we can finish the proof of part (b). The difference from part (a) is that in this case, the fact that the limit space is $\mb{R}^n$ will follow from Cheeger-Gromoll splitting theorem as argued in \cite{And}. Indeed, by Theorem \ref{thm compactness} and the equation for the Ricci curvature tensor in harmonic radius, the limit space is Ricci flat. On the other hand, the injectivity radius becomes infinity after blowing up. Hence, Cheeger-Gromoll splitting theorem can be applied.

\qed \\

Following the arguments in \cite{ChCo2}, one may also show that under condition \eqref{basic assumption}, the codimension of the singular space of the Gromov-Hausdorff limit is still at least 2 (see Theorem 5.1 in \cite{WZ}). Combining this result with Theorem \ref{thm harmonic radius bound}, we have

\begin{theorem}[$\e$-regularity]\label{e regularity}
Given $\rho>0$ and $q>2n$, for each $\e>0$, there is a $\delta=\delta(\e~|n,\rho,q)$ such that if $(\M, g)$ is a Riemannian manifold with $|Ric+\d^2L|\leq (n-1)\delta^2$, $|\d L|\leq \delta$, and $\vol(B_{10}(p))\geq \rho$, and
\[
d_{GH}(B_2(p), B_2((0^{n-1},x^*)))\leq \e,
\]
where $(0^{n-1},x^*)\in \mathbb{R}^{n-1}\times X$ for some metric space $X$, then the $W^{1,q}$ harmonic radius $r_h(p)$ satisfies
\[
r_h(p)\geq 1.
\]

\end{theorem}

\section{The Transformation Theorem}

In this and next section, following the guidelines in \cite{ChNa2}, we prove the Transformation and Slicing Theorems, which allow us to derive the Codimension 4 Theorem by following the remaining arguments as in \cite{ChNa2}. However, since our assumption is made on the Bakry-\'Emery Ricci curvature, to be able to overcome some technical difficulties, we need to add a weight to the concepts used in \cite{ChNa2}. We start by restating the definition of $\e$-splitting map introduced in \cite{ChNa2}.

\begin{definition}
\label{def e splitting}
A harmonic map $u=(u^1, u^2,\cdots, u^k):\ B_r(x)\rightarrow \mathbb{R}^k$ is an $\e$-splitting map, if\\
(1) $|\d u|\leq 1+\e$ in $B_r(x)$;\\
(2) $\displaystyle \fint_{B_r(x)}\left|<\d u^i, \d u^j>-\delta_{ij}\right|^2\leq \e^2$, $\forall i,j$;\\
(3) $\displaystyle r^2\fint_{B_r(x)}|\d^2 u^i|^2\leq \e^2$, $\forall i$.
\end{definition}

Denote by $\Delta_L:=\Delta-\d L\cdot \d$ the drifted Laplacian by the vector field $\d L$, $dV_L:=e^{-L}dV$ the weighted volume form, $ \vol_L(B_r(x)):=\int_{B_r(x)}dV_L$ the weighted volume of the geodesic ball $B_r(x)$, and $\fint^L_{B_r(x)} \cdots :=\frac{1}{\vol_L(B_r(x))}\int_{B_r(x)}\cdots dV_L$ the weighted average value over the ball $B_r(x)$.

In the definition above, using the drifted Laplacian and weighted average value instead of the regular ones, we define
\begin{definition}\label{def drifted e splitting}
A map $f=(f^1,f^2,\cdots,f^k): B_r(x)\rightarrow \mb{R}^k$ is called an $L$-harmonic map, if $\Delta_L f^i=0$, for $i=1,2,\cdots,k$.

Moreover, an $L$-harmonic map $f: B_r(x)\rightarrow \mb{R}^k$ is called an $L$-drifted $\e$-splitting map if\\
(1') $|\d f|\leq 1+\e$ in $B_r(x)$;\\
(2') $\displaystyle \fint^L_{B_r(x)}\left|<\d f^i, \d f^j>-\delta_{ij}\right|^2\leq \e^2$, $\forall i,j$;\\
(3') $\displaystyle r^2\fint^L_{B_r(x)}|\d^2 f^i|^2\leq \e^2$, $\forall i$.
\end{definition}

In the following, we first prove that the concepts of $\e$-splitting and $L$-drifted $\e$-splitting maps are equivalent. This equivalence will be used in the proof of Theorem \ref{codimension 4}.

\begin{lemma}\label{equiv. splitting maps}
Given $\rho>0$. For each $\e>0$ there exists an $\delta=\delta(\e~|n,\rho)$ satisfying the following property. Suppose that a manifold $\M^n$ satisfies $Ric+\d^2 L\geq -(n-1)\delta$, $|\d L|\leq \delta$,  and $\vol(B_{1}(x))\geq \rho$. Then for any $r\leq 1$, and an $\e$-splitting map $u$ on $B_r(x)$, there is an $L$-drifted $C\e^{1/2}$-splitting map $f$ on $B_{\frac{1}{4}r}(x)$ for some constant $C=C(n,\rho)$, and the converse is also true.

The notation $\delta(\e~|n,\rho)$ means a constant depending on the parameters in the parenthesis and $\delta\rightarrow 0$ as $\e\rightarrow0$.
\end{lemma}

\proof Suppose that $u$ is an $\e$-splitting map on $B_r(x)$. Without loss of generality, we may assume that $\e\leq 1$. Let $h^i$, $i=1,2,\cdots,k$, be the solution of the Dirichlet problem
\be\label{Delta L 1}
\Delta_L f^i=0\ \in \ B_r(x);\quad f^i=u^i\ on\ \pa B_{r}(x).
\ee

Since $u^i$ is a harmonic function, the function $h^i=f^i-u^i$ satisfies
\be\label{Delta L 2}
\Delta_L (h^i)=\d L\d u^i\ \in \ B_r(x);\quad h^i=0\ on\ \pa B_{r}(x).
\ee
Observe that $\Delta_L=e^L\textrm{div}(e^{-L}\d)$ and we can assume $L$ is locally bounded by replacing $L(\cdot)$ by
$L(\cdot)-L(x)$, it is well known that the integral maximum principle (or mean value property) and gradient estimate still hold for equation \eqref{Delta L 2} (see e.g. \cite{WZ}).

From the assumption on $u$, we have $|\d u|\leq 1+\e$. Combining this with $|\d L|\leq \delta$ and using the maximum principle, we get that for some $q>n/2$,
\[
\sup_{B_r(x)}|h^i|\leq Cr^2\left(\fint_{B_r(x)}|\d L \d u^i|^q\right)^{1/q}\leq C\delta r^2.
\]
Then it follows from the gradient estimate that
\be\label{d h}
\sup_{B_{\frac{1}{2}r}(x)}|\d h^i|^2\leq C\left[r^{-2}\fint_{B_r(x)}|h^i|^2 + \left(\fint_{B_r(x)}|\d L\d u^i|^{2q}\right)^{1/q}\right]\leq C\delta^2,
\ee
i.e.,
\be\label{1'}
\sup_{B_{\frac{1}{2}r}(x)}|\d f^i|\leq 1+\e+C\delta.
\ee
Also, from \eqref{d h}, (1) and (2) in Definition \ref{def e splitting}, and the boundedness of $L$, one has
\be\label{2'}
\al
&\fint^L_{B_{\frac{1}{2}r}(x)}|<\d f^i, \d f^j>-\delta_{ij}|\\
\leq& \fint^L_{B_{\frac{1}{2}r}(x)}|<\d h^i, \d u^j>|+|<\d u^i, \d h^j>| + |<\d h^i, \d h^j>| + |<\d u^i, \d u^j>-\delta_{ij}|\\
\leq& C(\delta+\e).
\eal
\ee
Now, let $\phi$ be  a cut-off function supported in $B_{\frac{1}{2}r}(x)$ with $\phi=1$ in $B_{\frac{1}{4}r}(x)$ and $|\d \phi|^2+|\Delta \phi|\leq \frac{C}{r^2}$ (See Lemma 1.5 in \cite{WZ}). It is straightforward to check that for the drifted Laplacian we have the following Bochner's formula.
\[
\Delta_L|\d F|^2=2|\d^2 F|^2+2<\d \Delta_L F, \d F> + 2(Ric+\d^2L)(\d F, \d F).
\]
Setting $F=f^i$, it implies that (see e.g. page 13 in \cite{WZ})
\be\label{3'}
\al
r^2\fint^L_{B_{\frac{1}{4}r}(x)}|\d^2 f^i|^2\leq & r^2\fint^L_{B_{\frac{1}{2}r}(x)} \phi|\d^2 f^i|^2\\
\leq& r^2\fint^L_{B_{\frac{1}{2}r}(x)} \frac{1}{2}\phi\Delta_L|\d f^i|^2 + (n-1)\delta|\d f^i|^2\\
=& C\delta +  \frac{1}{2}r^2\fint^L_{B_{\frac{1}{2}r}(x)}(|\d f^i|^2-1)\Delta_L\phi\\
\leq& C\delta + C(1+\delta)\fint^L_{B_{\frac{1}{2}r}(x)}\left||\d f^i|^2-1\right|\\
\leq& C(\delta+\e).
\eal
\ee
Here, in the last step, we have used \eqref{2'}.

Combining \eqref{1'}, \eqref{2'}, and \eqref{3'}, we have shown that $f$ is an $L$-drifted $C\e^{1/2}$-splitting map on $B_{\frac{1}{4}r}(x)$ for sufficiently small constant $\delta$.\qed\\

Next, recall the concept of the singular scale in \cite{ChNa2}:

\begin{definition}
Let $u: B_2(p)\rightarrow \mb{R}^k$ be a harmonic map. For $x\in B_1(p)$, $\delta>0$, the singular scale $s_x^{\delta}\geq0$ is the infimum of radii $s$ such that for all $s\leq r\leq \frac{1}{4}$ and all $1\leq l\leq k$, we have
\be\label{singular scale}
r^2\fint_{B_r(x)}|\Delta|\tilde{w}^l||\leq \delta\fint_{B_r(x)}|\tilde{w}^l|,
\ee
where $\tilde{w}^l=du^1\wedge du^2\wedge\cdots\wedge du^l$.
\end{definition}

Replacing harmonic map and Laplacian $\Delta$ above by $L$-harmonic map and the drifted Laplacian $\Delta_L$, we define similarly

\begin{definition}
Let $f: B_2(p)\rightarrow \mb{R}^k$ be an $L$-harmonic map. For $x\in B_1(p)$, $\delta>0$, the $L$-singular scale $s_{L,x}^{\delta}\geq0$ is the infimum of radii $s$ such that for all $s\leq r\leq \frac{1}{4}$ and all $1\leq l\leq k$, we have
\be\label{L singular scale}
r^2\fint^L_{B_r(x)}|\Delta_L|w^l||\leq \delta\fint^L_{B_r(x)}|w^l|,
\ee
where $w^l=df^1\wedge df^2\wedge\cdots\wedge df^l$.
\end{definition}

In the proofs of the Transformation and Slicing Theorems, we will use $L$-singular scale, but return to $\e$-splitting maps at the end when . Now, we are ready to state the Transformation Theorem, whose proof essentially follows the idea of \cite{ChNa2}. But for the purpose of deriving the higher order estimates as in Theorem 1.26 in \cite{ChNa2}, we first need to work with the drifted Laplacian and $L$-drifted $\e$-splitting maps, and prove certain transformation theorem under this weighted setting. Then come back to the regular Laplacian and $\e$-splitting maps by using the equivalence between $\e$-splitting maps and drifted $\e$-splitting maps in Lemma \ref{equiv. splitting maps}.

It seems that by using a Green's function argument instead of the heat kernel argument in \cite{ChNa2} for Claim 3 below, and adapting an argument in \cite{Bam}, the original proof can be shortened. Moreover, a uniformly positive lower bound of the diagonal entries of the matrices in the conclusion is obtained. More precisely, we have

\begin{theorem}[Transformation Theorem]\label{transformation theorem}
Given $\rho>0$. For every $\e>0$, there exists a $\delta=\delta(\e~|n,\rho)>0$ with the following property. Suppose that a manifold $\M^n$ satisfies $\left|Ric+\d^2 L\right|\leq (n-1)\delta$ with $|\d L|\leq \delta$, and $\vol(B_{10}(p))\geq \rho$, and let $f: B_{2}(p)\rightarrow \mb{R}^k$ be an $L$-drifted $\delta$-splitting map. Then

a) for any $x\in B_1(p)$ and $r\in[s_{L,x}^{\delta},\frac{1}{4}]$, there exists a lower triangular matrix $A=A(x,r)$ with positive diagonal entries so that $A\circ f: B_r(x)\rightarrow \mb{R}^k$ is an $L$-drifted $\e$-splitting map;

b) there is a constant $c_0=c_0(n)>0$, such that for any matrix $A(x,r)=(a_{ij})$ above, we have
\be\label{transformation eq1}
a_{ii}\geq c_0, \ 1\leq i\leq k.
\ee

\end{theorem}

\begin{proof}
Following \cite{ChNa2}, we prove by induction on $k$. Unless otherwise specified, the letter $C$ always denotes some constant depending on $n$, $\lambda$ and $\rho$. First of all, the proof of the theorem when $k=1$ is analogous to the proof of lemma 3.34 in \cite{ChNa2}. By using the Bochner's formula, we get
\be\label{L Bochner}
\Delta_L|\d f|=\frac{|\d^2 f|^2-\left|\d|\d f|\right|^2}{|\d f|}+\frac{(Ric+\d^2L)(\d f,\d f)}{|\d f|}.
\ee
Notice that since $\Delta f = <\d L, \d f>$, the improved Kato's inequality becomes
\[
|\d|\d f||^2\leq \frac{2n-1}{2n-2}|\d^2 f|^2 + |\d L|^2|\d f|^2.
\]
Thus, it follows from \eqref{L Bochner} that
\[
\Delta_L|\d f|\geq \frac{1}{2n-2}\frac{|\d^2 f|^2}{|\d f|} - C\delta |\d f|.
\]
Then using \eqref{L singular scale} gives
\[
r^2\fint^L_{B_{2r}(x)}\frac{|\d^2 f|^2}{|\d f|}\leq C\delta\fint^L_{B_{2r}(x)}|\d f|,
\]
and hence,
\[
r\fint^L_{B_{2r}(x)} |\d^2 f|\leq \left(r^2\fint^L_{B_{2r}(x)}\frac{|\d^2 f|^2}{|\d f|}\right)^{1/2}\left(\fint^L_{B_{2r}(x)}|\d f|\right)\leq C\delta^{1/2}\fint^{L}_{B_{2r}(x)}|\d f|.
\]
Thus, by setting $v=f\bigg/\left(\fint^L_{B_r(x)} |\d f|\right)$, we may proceed as in Lemma 3.34 in \cite{ChNa2}. Here notice that the heat kernel Gaussian bounds was used in the proof of Lemma 3.34 in \cite{ChNa2}. In our case, it is well known that the Gaussian bounds of the heat kernel and Green's functions estimates for the drifted Laplacian $\Delta_L$ are still valid, since both $|\d L|$ and $|L|$ are bounded. Or instead, one can use the mean value property.\\

Now suppose that the theorem holds for $k-1$ and fails for $k$. Then there exists an $\e>0$ such that for some $\delta_j\rightarrow 0$, there is a sequence of pointed manifolds $(\M^n_j,g_j,p_j)$ and smooth functions $\{L_j\}$ with $$\left|Ric_{\M_j}+\d^2 L_j\right|\leq (n-1)\delta_j,\quad |\d L_j|\leq \delta_j,\quad \vol(B_{10}(p_j))\geq \rho,$$
and $L_j$-drifted $\delta_j$-splitting maps $f_j: B_2(p_j)\rightarrow \mb{R}^k$ together with points $x_j\in B_1(p_j)$ and $r_j\in[s_{L_j,x_j}^{\delta_j},\frac{1}{4}]$, such that there is no lower triangular matrix $A$ with positive diagonal entries so that $A\circ f_j: B_{r_j}(x_j)\rightarrow R^k$ is $L_j$-drifted $\e$-splitting.

Notice that $r_j\rightarrow 0$. Indeed, if $r_j\geq c>0$, then since $r_j\leq 1/4$, we have $B_{c}(x_j)\subseteq B_{r_j}(x_j)\subseteq B_{5/4}(p_j)\subseteq B_{3/2}(x_j)\subseteq B_2(p_j)$, which means the sizes of all these balls are comparable. Then the fact that $f_j: B_2(p_j)\rightarrow \mb{R}^k$ is $L_j$-drifted $\delta_j$ splitting and the volume doubling property immediately implies that $f_j: B_{r_j}(x_j)\rightarrow \mb{R}^k$ is an $L_j$-drifted  $C\delta_j$-splitting map, which in particular is an $L_j$-drifted $\e$-splitting when $j$ is big enough, and hence contradicts to the hypothesis above.

Thus, we may assume that $r_j$ is the supremum of the radii for which $A\circ f_j: B_{r_j}(x_j)\rightarrow R^k$ is not an $L_j$-drifted $\e$-splitting map for any lower triangular matrix $A$. It then follows that there exists a lower triangular matrix $A_j$ such that $A_j\circ f_j: B_{2r_j}(x_j)\rightarrow \mb{R}^k$ is an $L_j$-drifted $\e$-splitting map. Moreover, since $|\d L_j|$ is bounded, by replacing $L_j$ by $L_j-L_j(x_j)$ whenever necessary, we may assume that $|L_j|$ is bounded in $B_{1}(x_j)$.

Let
\be\label{vj}
v_j=r_j^{-1}A_j\circ(f_j-f_j(x_j)),
\ee
and use the rescaled metric $g'_j=r_j^{-2}g_j$ for the following arguments. Then $v_j: B_2(x_j)\rightarrow \mb{R}^k$ is $L_j$-drifted $\e$-splitting, and for any $2\leq r\leq \frac{1}{4}r_j^{-1}$, there is a lower triangular matrix $A_r$ with positive diagonal entries such that $A_r\circ v_j: B_r(x_j)\rightarrow \mb{R}^k$ is $L_j$-drifted $\e$-splitting.\\

The following Claim 1 and 2 are directly from \cite{ChNa2} (see pages 1118-1121 for proofs). The only change caused by the drifted situation is that the volume element $dV$ becomes $dV_{L_j}$.\\

{\bf Claim 1:} For any $2\leq r\leq \frac{1}{4}r_j^{-1}$, one has
\[
(1-C\e)A_{2r}\leq A_r\leq (1+C\e)A_{2r},
\]
which implies that for any $1\leq a,l\leq k$,
\be\label{claim1 eq1}
\sup_{B_r(x_j)} |\d v_j^a|\leq (1+C\e)r^{C\e},
\ee
\be\label{claim1 eq2}
\sup_{B_r(x_j)} |w_j^l|\leq (1+C\e)r^{C\e},
\ee
\be\label{claim1 eq3}
r^2\fint^{L_j}_{B_r(x_j)}|\d^2 v_j^a|^2\leq C\e r^{C\e}.
\ee

{\bf Claim 2:} There exists a lower triangular matrix $A$ with positive diagonal entries such that $|A-I|\leq C\e$, $A\circ v_j: B_2(x_j)\rightarrow \mb{R}^k $ is $L_j$-drifted $C\e$-splitting, and for each $R>0$, after discarding the last component, the map $A\circ v_j: B_R(x_j)\rightarrow \mb{R}^{k-1}$ is $L_j$-drifted $\e_j(R)$-splitting. Here $\e_j(R)\rightarrow 0$ whenever $R$ is fixed.

From now on, let $v_j$ represents $A\circ v_j$ in claim 2. Thus, as shown in (3.61) and (3.63) in \cite{ChNa2}, we have for any $2\leq r\leq \frac{1}{4}r_j^{-1}$ and $1\leq l\leq k$ that
\be\label{claim2 eq1}
r^2\fint^{L_j}_{B_r(x_j)}|\d w_j^l|^2\leq C\e r^{C\e},
\ee
\be\label{claim2 eq2}
r^2\fint^{L_j}_{B_r(x_j)}\left|\Delta_L|w_j^l|\right|\leq C\delta_j r^{C\e},
\ee
where $w_j^l=dv_j^1\wedge dv_j^2\wedge\cdots\wedge dv_j^l$.

From Claim 2, we know that $(v_j^1,\cdots,v_j^{k-1})$ is $L_j$-drifted $\e_j$-splitting on $B_1(x_j)$. To get a contradiction, we also need to show that after transformation, the average of $|dv_j^k|^2$ is approaching $1$, and $dv_j^k$ and $dv_j^1,\cdots,dv_j^{k-1}$ tend to be orthogonal.

To show this, we first show that the standard deviation of $|dv_j^k|^2$ and $<dv_j^a,dv_j^k>$ ($1\leq a\leq k-1$) are approaching $0$ on scale larger than $1$ (Claims 3 and 4 below) similar to \cite{ChNa2}. However, we use another approach to prove these claims. For Claim 3 below, instead of using the heat kernel, the proof uses an argument involving Green's function.

{\bf Claim 3:} For any $R\ge 1$, we have
\be\label{claim3 eq1}
\fint^{L_j}_{B_R(x_j)} \left||w_j^l|^2-\fint^{L_j}_{B_R(x_j)}|w_j^l|^2\right|\leq \e_j(R),\ \forall 1\leq l\leq k,
\ee
and
\be\label{claim3 eq2}
\fint^{L_j}_{B_R(x_j)}\left|<d v_j^a,d v_j^k>-\fint^{L_j}_{B_R(x_j)}<d v_j^a,d v_j^k>\right|\leq \e_j(R),\ \forall 1\leq a\leq k-1.
\ee

{\it Proof of Claim 3:} Fix an $R\geq 1$. For any $x\in B_R(x_j)$ and $1\leq l\leq k$, let
\[
M^R(x)=\sup_{r\leq R}\fint^{L_j}_{B_r(x)} \left|\Delta_L |w_j^l|\right|.
\]
Then as in (3.65) in \cite{ChNa2}, since we have \eqref{claim2 eq2}, by the maximal function arguments, there exists a subset $U_j\subseteq B_R(x_j)$ satisfying
\be\label{U_j}
\frac{\vol_{L_j}(B_R(x_j)\setminus U_j)}{\vol_{L_j}(B_R(x_j))}\leq \e_j(R),
\ee
\be\label{maximal function}
M^R(x)\leq \e_j(R),\ \forall x\in U_j.
\ee

To get \eqref{claim3 eq1}, it suffices to show that
\be\label{claim3 eq3}
\left||w_j^l|^2(x)-|w_j^l|^2(y)\right|\leq \e_j(R), \ \forall x,y\in U_j,
\ee
because it will then follow that
\[
\al
&\fint^{L_j}_{B_R(x_j)}\left||w_j^l|^2-\fint_{B_R(x_j)}|w_j^l|^2\right|\\
\leq& \fint^{L_j}_{B_R(x_j)}\left||w_j^l|^2(y)-|w_j^l|^2(x)\right|dy + \left|\fint^{L_j}_{B_R(x_j)}|w_j^l|^2(x)-|w_j^l|^2(z)dz\right|\\
\leq& \frac{2}{\vol_{L_j}(B_R(x_j))}\left[\int_{U_j}+\int_{B_R(x_j)\setminus U_j}\right]\left||w_j^l|^2(y)-|w_j^l|^2(x)\right|e^{-L_j}dy\\
\leq& \e_j(R)+C\frac{\vol_{L_j}(B_R(x_j)\setminus U_j)}{\vol_{L_j}(B_R(x_j))}\\
\leq& \e_j(R).
\eal
\]

Since $|w_j^l|\leq C(n)$, to show \eqref{claim3 eq3}, we only need to show
\[
\left||w_j^l|(x)-|w_j^l|(y)\right|\leq \e_j(R),\ \forall x,y\in U_j.
\]
First, choose a cut-off function $\phi$ such that $\phi=1$ in $B_{\frac{1}{8}r_j^{-1}}(x_j)$, $\phi=0$ in $B^c_{\frac{1}{4}r_j^{-1}}(x_j)$, and
\be\label{cutoff function}
|\d \phi|^2+|\Delta \phi|\leq Cr_j^2.
\ee

Denote by $G_{L_j}(x,y)$ the Green's function for the drifted Laplacian $\Delta_{L_j}$ on $\M_j$. Since $|L_j|$ is bounded on $B_{r_j^{-1}}(x_j)$, for the Green's function for $\Delta_{L_j}$, we still have that
\be\label{Green's function estimate}
|G_{L_j}(x,y)|\leq \frac{C}{d(x,y)^{n-2}},\ and\ |\d_y G_{L_j}(x,y)|\leq \frac{C}{d(x,y)^{n-1}}, \qquad x, y \in
B_{\frac{1}{4}r_j^{-1}}(x_j).
\ee

Without loss of generality, we may assume that $R\leq \frac{1}{16}r_j^{-1}$. Thus, for $x,y\in U_j$.
\[
\al
&\left||w_j^l|(x)-|w_j^l|(y)\right|\\
&=\left|\phi|w_j^l|(x)-\phi|w_j^l|(y)\right|\\
&=\left|\int_{\M_j}\left(G_{L_j}(x,z)-G_{L_j}(y,z)\right)\Delta_{L_j}(\phi|w_j^l|)~e^{-L_j}dz\right|\\
&\leq \int_{\M_j}\left|G_{L_j}(x,z)-G_{L_j}(y,z)\right||\Delta_{L_j} \phi||w_j^l|e^{-L_j}dz+
2\int_{\M_j}\left|G_{L_j}(x,z)-G_{L_j}(y,z)\right||\d \phi|\left|\d|w_j^l|\right|e^{-L_j}dz\\
&\quad +\int_{\M_j}\left|G_{L_j}(x,z)-G_{L_j}(y,z)\right|\phi\left|\Delta_{L_j}|w_j^l|\right|e^{-L_j}dz\\
&: =I+II+III.
\eal
\]
Using \eqref{claim1 eq2}, \eqref{cutoff function} and \eqref{Green's function estimate}, we have
\[
\al
I&\leq \int_{B_{\frac{1}{4}r_j^{-1}}(x_j)\setminus B_{\frac{1}{8}r_j^{-1}}(x_j)}|\d G_{L_j}(x^*,z)|\cdot d(x,y)\cdot Cr_j^{2-2C\e}e^{-L_j}dz \leq \frac{CR}{r_j^{1-n}}\cdot r_j^{2-2C\e}\cdot \vol_{L_j}(B_{\frac{1}{4}r_j^{-1}}(x_j))\\
&\leq CRr_j^{1-C\e}\leq \e_j(R).
\eal
\]
Similarly, from \eqref{claim2 eq1}, \eqref{cutoff function} and \eqref{Green's function estimate}, one gets $II\leq \e_j(R)$.

Finally, one has
\[
\al
III&\leq \int_{B_{2R}(x_j)}\left(|G_{L_j}(x,z)|+|G_{L_j}(y,z)|\right)\left|\Delta_{L_j}|w_j^l|\right|e^{-L_j}dz\\
&\quad +\int_{M_j\setminus B_{2R}(x_j)}\left(|G_{L_j}(x,z)-G_{L_j}(y,z)|\right)\left|\Delta_{L_j}|w_j^l|\right|\phi e^{-L_j}dz\\
&: =(1)+(2),
\eal
\]
where, by \eqref{maximal function} and \eqref{Green's function estimate},
\[
\al
(1)&\leq \sum_{k=-\infty}^2\int_{B_{2^kR}(x)\setminus B_{2^{k-1}R}(x)}|G_{L_j}(x,z)|\left|\Delta_{L_j}|w_j^l|\right|e^{-L_j}dz\\
&\quad +\sum_{k=-\infty}^2\int_{B_{2^kR}(y)\setminus B_{2^{k-1}R}(y)}|G_{L_j}(y,z)|\left|\Delta_{L_j}|w_j^l|\right|e^{-L_j}dz\\
&\leq \sum_{k=-\infty}^2 \frac{C\vol_{L_j}(B_{2^kR}(x))}{\left(2^{k-1}R\right)^{n-2}}\fint^{L_j}_{B_{2^kR}(x)}\left|\Delta_{L_j}|w_j^l|\right|~dz+ \sum_{k=-\infty}^2 \frac{C\vol_{L_j}(B_{2^kR}(y))}{\left(2^{k-1}R\right)^{n-2}}\fint^{L_j}_{B_{2^kR}(y)}\left|\Delta_{L_j}|w_j^l|\right|~dz\\
&\leq CR^2\e_j(R)\sum_{k=-\infty}^2 2^{2k}\leq \e_j(R),
\eal
\]
while from \eqref{claim2 eq2}, we have
\[
\al
(2)&\leq \sum_{k=2}^{\infty}\int_{B_{2^kR}(x_j)\setminus B_{2^{k-1}R}(x_j)}|\d G_{L_j}(x^*,z)|d(x,y)\left|\Delta_{L_j}|w_j^l|\right|\phi e^{-L_j}dz\\
&\leq \sum_{k=2}^{\infty}\frac{CR\vol_{L_j}(B_{2^kR}(x_j))}{\left(2^{k-2}R\right)^{-1}}\fint^{L_j}_{B_{2^kR}(x_j)}\left|\Delta_{L_j}|w_j^l|\right|\phi~dz\\
&\leq C\delta_j\sum_{k=2}^{\infty}2^{-k}\cdot(2^kR)^{C\e}\leq \e_j(R).
\eal
\]
Therefore, we get $III\leq \e_j(R)$, and this finishes the proof of \eqref{claim3 eq1}.

The proof of \eqref{claim3 eq2} is similar. Firstly, notice that for the maximal function argument to work, by Claim 2, one has for any $1\leq a\leq k-1$ and $2\leq r\leq \frac{1}{2}r_j^{-1}$,
\be\label{higher order estimate}
\al
r^2\fint^{L_j}_{B_r(x_j)}\left|\Delta_{L_j}<d v_j^a,d v_j^k>\right|
=&r^2\fint^{L_j}_{B_r(x_j)}\left|2<\d^2 v_j^a, \d^2 v_j^k> + 2(Ric_{\M_j}+\d^2 L_j)(\d v_j^a, \d v_j^k)\right|\\
\leq& 2\left[r^2\fint^{L_j}_{B_r(x_j)}|\d^2 v_j^a|^2\right]^{1/2}\left[r^2\fint^{L_j}_{B_r(x_j)}|\d^2 v_j^k|^2\right]^{1/2}+C\e\delta_jr_j^2r^2\\
\leq& \e_j(r).
\eal
\ee
Thus, an analogue of the proof of \eqref{claim3 eq3} gives
\be\label{claim3 eq4}
|<d v_j^a, d v_j^k>(x)-<d v_j^a, d v_j^k>(y)|\leq \e_j(R),\ \forall x,y\in U_j.
\ee
For any $x\in U_j$, by \eqref{U_j}, \eqref{maximal function} and \eqref{claim3 eq4}, we have
\[
\al
&\fint^{L_j}_{B_R(x_j)}\left|<d v_j^a, d v_j^k>(x)- <d v_j^a, d v_j^k>(z)\right|dz\\
\leq& \frac{1}{\vol_{L_j}(B_R(x_j))}\left[\int_{U_j}+\int_{B_R(x_j)\setminus U_j}\right]\left|<d v_j^a, d v_j^k>(x)-<d v_j^a, d v_j^k>(z)\right| e^{-L_j}dz\\
\leq& \e_j(R).
\eal
\]
Thus, it follows that
\be\label{claim4 eq3}
\al
&\fint^{L_j}_{B_R(x_j)}\left|<d v_j^a, d v_j^k>-\fint^{L_j}_{B_R(x_j)}<d v_j^a, d v_j^k>\right|\\
\leq& \fint^{L_j}_{B_R(x_j)}\left|<d v_j^a, d v_j^k>(y)-<d v_j^a, d v_j^k>(x)\right|dy\\
&\ +\fint^{L_j}_{B_R(x_j)}\left|<d v_j^a, d v_j^k>(x)- <d v_j^a, d v_j^k>(z)\right|dz\\
\leq& \e_j(R).
\eal
\ee
This finishes the proof of Claim 3.\\

Next, we follow a method in \cite{Bam} to show Claim 4 below and derive the contradiction.

{\bf Claim 4:} For any $R\geq 1$, we have
\be\label{claim4 eq1}
\fint^{L_j}_{B_R(x_j)}\left||d v_j^k|^2-\fint^{L_j}_{B_R(x_j)}|d v_j^k|^2\right|\leq \e_j(R).\\
\ee
The details of the proof of the Claim was not given in \cite{Bam}. For readers' convenience, we give a proof in Appendix A.

Now similar to the arguments on page 101 in \cite{Bam}, let
\[
a_j^l=-\fint^{L_j}_{B_2(x_j)}<dv_j^l, dv_j^k>\bigg/\fint^{L_j}_{B_2(x_j)}|d v_j^l|^2, \ \forall 1\leq l\leq k-1,
\]
\[
\tilde{v}_j^l=v_j^l,\quad \tilde{v}_j^k=v_j^k+\sum_{l=1}^{k-1}a_j^lv_j^l.
\]

and
\[ \hat{v}_j^l=\tilde{v}_j^l,\quad \hat{v}_j^k=\tilde{v}_j^k\Big/\fint^{L_j}_{B_2(x_j)}|d\tilde{v}_j^k|^2.\] 

Then $(\hat{v}_j^1,\cdots,\hat{v}_j^k): B_1(x_j)\rightarrow \mb{R}^k$ is an $L_j$-drifted $\e_j$-splitting map, which contradicts to the inductive hypothesis when $j$ is sufficiently large. The details can also be found in the Appendix A.

Hence, this finishes the proof of part a).\\

To show b), denote by $v=(v^1,\cdots,v^k)=A\circ f$ the $L$-drifted $\e$-splitting map from $B_r(x)$ to $\mb{R}^k$. From the definition, we have
\[
\fint^L_{B_r(x)}\left||dv^1|^2-1\right|\leq \e^2,
\]
which together with the fact that $f$ is an $L$-drifted $\delta$-splitting map, implies that
\[
1-\e^2\leq a_{11}^2\fint^L_{B_r(x)}|df^1|^2\leq a_{11}^2(1+C\delta),
\]
i.e.,
\[
a_{11}\geq \frac{1}{2}.
\]

Similarly, since
\[
v^2=a_{21}f^1+a_{22}f^2=\frac{a_{21}}{a_{11}}v^1+a_{22}f^2,
\]
we have
\[
\al
a_{22}^2(1+C\delta)&\geq\fint^L_{B_r(x)}|a_{22}d f^2|^2\\
&= \fint^L_{B_r(x)}|dv^2|^2+\frac{a_{21}^2}{a_{11}^2}\fint^L_{B_r(x)}|dv^1|^2-2\frac{a_{21}}{a_{11}}\fint^L_{B_r(x)}<dv^1,dv^2>\\
&\geq 1-\e^2 + (1-\e^2)\frac{a_{21}^2}{a_{11}^2}-\e^2\left|\frac{a_{21}}{a_{11}}\right|\\
&\geq c_1.
\eal
\]
Obviously, when $\delta$ is chosen small enough so that $C\delta<1$, then we get $a_{22}\geq c_0$.

In general, notice that
\[
a_{ll}f^l =v^l - a_{l1}f^1-a_{l2}f^2-\cdots-a_{l(l-1)}f^{l-1}=v^l - \eta_1v^1-\eta_2v^2-\cdots-\eta_{l-1}v^{l-1},
\]
where $\eta_i$'s are constants depending on the entries $a_{ij}$, $1\leq i,j\leq l$.

Since $dv^1, \cdots, dv^{l-1}, dv^l$ are almost orthonormal under the inner product $\fint^L_{B_r(x)}<\cdot,\cdot>$, it is not hard to see that
\[
\fint^L_{B_r(x)}a_{ll}^2|df^l|^2=\fint^L_{B_r(x)}\left|v^l - \eta_1v^1-\eta_2v^2-\cdots-\eta_{l-1}v^{l-1}\right|^2\geq c_1,
\]
regardless the values of $\eta_1,\cdots,\eta_{l-1}$.
Thus, we get
\[
a_{ll}\geq c_0,
\]
due to the fact that $|df^l|^2\leq 1+C\delta$.

The proof of the theorem is completed.
\end{proof}



\section{The Slicing Theorem and Proof of Theorem \ref{codimension 4}}

Using the Transformation Theorem, we are able to prove the Slicing Theorem. But before that, we need two more lemmas. Assume that $f:B_2(p)\rightarrow \mb{R}^k$ is an $L$-drifted $\delta$-splitting map. For any open set $U$ and $1\leq l\leq k$, define measure
\be\label{measure mu}
\mu_L^l(U)=\int_{U}|w^l| dV_L\bigg/\int_{B_{\frac{3}{2}}(p)}|w^k|dV_L,
\ee
where $w^l=df^1\wedge df^2\wedge\cdots\wedge df^l$, and $dV_L=e^{-L}dV$ is the weighted volume element.

Using similar arguments as in Lemma 4.1 in \cite{ChNa2}, we can show a doubling property for $\mu_L^l$.

\begin{lemma}\label{lem measure doubling}
For any $x\in B_1(p)$, $s_{L,x}^{\delta}\leq r\leq 1/4$ and $1\leq l\leq k$, we have
\be\label{measure doubling}
\mu_L^l(B_{2r}(x))\leq C(n) \mu_L^l(B_r(x)).
\ee
\end{lemma}

Moreover, by \eqref{transformation eq1} Theorem \ref{transformation theorem} part (b), and following a similar proof, one can actually derive a slight more general result than Lemma 4.2 in \cite{ChNa2}, which is needed for completing the proof of the Slicing Theorem. More explicitly,

\begin{lemma}\label{lem measure comparison}
For any $x\in B_1(p)$, $s_{L,x}^{\delta}\leq r\leq 1/4$ and $1\leq l\leq k$, we have
\be\label{measure comparison}
\left|f(B_r(x))\right|\leq Cr^{-(n-k)} \mu_L^l(B_r(x)),
\ee
where $|f(B_r(x))|$ denotes the Euclidean measure of $f(B_r(x))\subseteq\mb{R}^k$.
\end{lemma}

\proof By Theorem \ref{transformation theorem}, there is a lower triangular matrix $A=(a_{ij})\in GL(k)$ with positive diagonal entries such that
\[
\bar{f}=A\circ f: B_{2r}(x)\rightarrow \mb{R}^k
\]
is an $\e$-splitting, and hence
\be\label{w'}
\fint^L_{B_{r}(x)}\left||\bar{w}^l|-1\right|\leq C\fint^L_{B_{2r}(x)}\left||\bar{w}^l|-1\right|\leq C\e,
\ee
where $\bar{w}^l=d\bar{f}^1\wedge\cdots\wedge d\bar{f}^l$.

Define
\[
\bar{\mu}^l_L(U)=\frac{\int_{U}|\bar{w}^l|dV_L}{\int_{B_{3/2}(p)}|w^k|dV_L}.
\]
Since $f$ is $L$-drifted $\delta$-splitting on $B_2(p)$, by the volume comparison, it is $L$-drifted $C\delta$-splitting on $B_{3/2}(p)$.

This together with \eqref{w'} implies
\be\label{mu'}
\al
\bar{\mu}^l_L(B_r(x))=&\frac{\int_{B_r(x)}|\bar{w}^l|dV_L}{\int_{B_{3/2}(p)}|w^k|dV_L}\\
\geq& (1-C\e)\frac{\vol_L(B_r(x))}{\vol_L(B_{3/2}(p))}\fint^L_{B_r(x)}|\bar{w}^l|\\
\geq& Cr^n.
\eal
\ee

On the other hand, since $|\d \bar{f}|\leq 1+\e$ in $B_{2r}(x)$, it is easy to check that
\[
\bar{f}(B_r(x))\subseteq B_{2r}(\bar{f}(x)).
\]
Thus,
\be\label{u'}
|\bar{f}(B_r(x))|\leq Cr^k.
\ee
Combining \eqref{mu'} and \eqref{u'}, we get
\be\label{u' mu'}
|\bar{f}(B_r(x))|\leq Cr^{-(n-k)}\bar{\mu}^l_L(B_r(x)).
\ee
Notice that
\be\label{det A}
|\bar{f}(B_r(x))|=det(A)|f(B_r(x))|,
\ee
and from \eqref{transformation eq1} in Theorem \ref{transformation theorem} part b), one has
\be\label{det A1}
\bar{\mu}^l_L=a_{11}\cdots a_{ll}\mu^l_L=\frac{det(A)\mu^l_L}{a_{l+1~l+1}\cdots a_{kk}}\leq c_0^{-(k-l)}det(A)\mu^l_L.
\ee
Plugging \eqref{det A} and \eqref{det A1} into \eqref{u' mu'}, the lemma follows immediately. \qed \\

To prove the Slicing theorem, we also need the following higher order integral estimates for a $\delta$-splitting map, the proof of which again is similar to Theorem 1.26 in \cite{ChNa2}.

\begin{theorem}\label{prop derivative estimates}
Given $\rho>0$. For each $\e>0$, there exists a $\delta_1=\delta_1(\e~|n,\rho)>0$ with the following property. Suppose that a manifold $\M^n$ satisfies $Ric+\d^2 L\geq -(n-1)\delta_1$ with $|\d L|\leq \delta_1$, and $\vol(B_{10}(p))\geq \rho$. Let $f:B_2(p)\rightarrow \mb{R}^k$ be an $L$-drifted $\delta$-splitting map. Then we have:\\
(1) There exists $\gamma(n,\rho)>0$ such that for each $1\leq l\leq k$,
\be\label{1}
\fint^L_{B_{3/2}(p)}\frac{|\d^2u^l|^2}{|\d u^l|^{1+\gamma}}\leq \e.
\ee
(2) For any $1\leq l\leq k$, the normal mass of $\Delta_L|w^l|$ satisfies
\be\label{higher order estimate1}
\fint^L_{B_{3/2}(p)}\left|\Delta_L|w^l|\right|\leq \e.
\ee
\end{theorem}
The higher order estimate \eqref{higher order estimate1} will be used in the proof of the Slicing Theorem below, whose proof follows from the Bochner's formula for the drifted Laplacian $\Delta_L$, similar to the proof of \eqref{higher order estimate}.

Now we are ready to prove the Slicing theorem

\begin{theorem}[Slicing Theorem]\label{slicing theorem}
Given $\rho>0$. For each $\e>0$, there exists a $\bar{\delta}=\bar{\delta}(\e~|n,\rho)>0$ such that the following is satisfied. Suppose that a manifold $\M^n$ satisfies $|Ric+\d^2 L|\leq (n-1)\bar{\delta}$, $|\d L|\leq \bar{\delta}$, and $\vol(B_{10}(p))\geq \rho$. Let $f: B_2(p)\rightarrow \mb{R}^{n-2}$ be an $L$-drifted $\delta$-splitting map. Then there is a subset $G_{\e}\subseteq B_1(0^{n-2})$ such that\\
(1) $\vol(G_{\e})\geq \vol(B_1(0^{n-2}))-\e$,\\
(2) $f^{-1}(s)\neq\emptyset$ for each $s\in G_{\e}$,\\
(3) for each $x\in f^{-1}(G_{\e})$ and $r\leq 1/4$, there is a lower triangular matrix $A$ with positive diagonal entries so that $A\circ f: B_r(x)\rightarrow \mb{R}^{n-2}$ is an $L$-drifted $\e$-splitting map.

\end{theorem}

\begin{proof}
Firstly, by a generalization of the results in section 2 in \cite{CCT} (see e.g. Lemma 5.7 in \cite{WZ}), we know that there exists a $\delta_2>0$ such that when the assumptions of the theorem are satisfied, we have
\be\label{slicing eq1}
\left|B_1(0^{n-2})\setminus f(B_1(p))\right|\leq \e/2.
\ee

Let $\delta$ be the parameter in the Transformation Theorem \ref{transformation theorem}. Set
\[
D_{\delta}=\bigcup_{x\in B_1(p),\  s_{L,x}^{\delta}>0}B_{s_{L,x}^{\delta}}(x).
\]
Next, we show that for $\bar{\delta}$ small enough, it holds that
\[
\left|f(D_{\delta})\right|\leq \e/2.
\]
Then, setting $G_{\e}=f(B_1(p))\setminus f(D_{\delta})$ will finish the proof of the theorem.

The collection of balls $\left\{B_{s_{L,x}^{\delta}}(x)\right\},\ x\in D_{\delta}$ forms a covering of $D_{\delta}$. Therefore, there exists a subcollection of mutually disjoint balls $\{B_{s_j}(x_j)\}$, where $s_j=s_{L,x_j}^{\delta}$, such that
\[
D_{\delta}\subseteq\bigcup_{j}B_{6s_j}(x_j).
\]
Since $s_j$ is the $L$-singular scale, the inequality \eqref{singular scale} reaches equality at $w^{l_j}$ for some $1\leq l_j\leq n-2$, i.e.,
\be\label{equality at singular scale}
s_j^2\fint^L_{B_{s_j}(x_j)}\left|\Delta_L|w^{l_j}|\right|=\delta\fint^L_{B_{s_j}(x_j)}|w^{l_j}|.
\ee
Moreover, we may assume that $\bar{\delta}$ is small enough so that $s_{L,x}^{\delta}\leq 1/32$. Then, by Lemma \ref{lem measure doubling} (see \eqref{measure doubling}), Lemma \ref{lem measure comparison} (see \eqref{measure comparison}), and \eqref{equality at singular scale}, we have
\be\label{uB}
\al
\left|f(D_{\delta})\right|\leq& \sum_j\left|f(B_{6s_j}(x_j))\right|\leq C\sum_j (6s_j)^{-2}\mu^{l_j}_L(B_{6s_j}(x_j))\\
\leq& C\sum_j s_j^{-2}\mu^{l_j}_L(B_{s_j})(x_j)=C\sum_j s_j^{-2}\frac{\int_{B_{s_j}(x_j)}|w^{l_j}|dV_L}{\int_{B_{3/2}(p)}|w^{n-2}|dV_L}\\
=&\frac{C\delta^{-1}}{\int_{B_{3/2}(p)}|w^{n-2}|dV_L}\sum_j\int_{B_{s_j}(x_j)} \left|\Delta_L|w^{l_j}|\right|dV_L.
\eal
\ee
From the fact that $f$ is $L$-drifted $\delta$-splitting on $B_2(p)$, we know that
\[
\fint^L_{B_{3/2}(p)}|w^{n-2}|\geq 1-C\delta.
\]
Putting this and the fact that $\{B_{s_j}(x_j)\}$ are disjoint into \eqref{uB}, we finally reach
\[
\al
\left|f(D_{\delta})\right|\leq& \frac{C\delta^{-1}}{\int_{B_{3/2}(p)}|w^{n-2}|dV_L}\sum_j \int_{B_{s_j}(x_j)}\sum_{l=1}^{n-2} \left|\Delta_L|w^{l}|\right|dV_L\\
\leq& \frac{C\delta^{-1}}{\int_{B_{3/2}(p)}|w^{n-2}|e^{-L}dV}\int_{B_{3/2}(p)}\sum_{l=1}^{n-2} \left|\Delta_L|w^{l}|\right|dV_L\\
\leq& C\delta^{-1}\fint^L_{B_{3/2}(p)}\sum_{l=1}^{n-2} \left|\Delta_L|w^{l}|\right|\\
\leq& \e/2.
\eal
\]
The last step above holds since we may choose $\delta_1=\delta_1(n, \frac{\e}{2}C^{-1}\delta)$ in Theorem \ref{prop derivative estimates}.

Therefore, setting $\bar{\delta}<\min(\delta_1,\delta_2,\delta)$ completes the proof.
\end{proof}

With the Slicing Theorem, we can finish the proof of Theorem \ref{codimension 4}.

\proof[Proof of Theorem \ref{codimension 4}:] Firstly, we need the lemma below to play the role of Lemma 1.21 in \cite{ChNa2}, which generalized the corresponding result in \cite{ChCo1} to the case where Bakry-\'Emery Ricci curvature has a lower bound.

\begin{lemma}\label{lem coordinate approximation}
Given $\rho>0$, for any $\epsilon$, there exist $\delta=\delta(\e~|n,\rho)>0$ such that the following holds. Assume that $\Ric+\d^2 L \geq - \delta g$, $|\d L|\leq \delta$, and
\be\label{noncollapsing}
 \vol (B_{10}(y)) \ge \rho>0, \quad \forall y \in \M.
\ee

(a) If
\[
d_{GH}(B_{\delta^{-1}}(p), B_{\delta^{-1}}((0^k,x^*))\leq \e,
\]
where $(0^k,x^*)\in \mb{R}^k\times C(X)$ with $x^*$ being the vertex of the metric cone $C(X)$ over some metric space $X$, then for any $R\leq 1$, there exists an $L$-drifted $\e$-splitting map $f=(f_1, f_2, \cdots, f_k):\ B_R(p)\rightarrow \mathbb{R}^k$.

(b) If $f=(f_1, f_2, \cdots, f_k):\ B_{8R}(p)\rightarrow \mathbb{R}^k$ is an $L$-drifted $\delta$-splitting map for $R\leq 1$, then there is a map $\Phi: B_R(x)\rightarrow f^{-1}(0)$ such that $(u,\Phi): B_R(p)\rightarrow B_{R}((0^{k},x^*))\subset \mb{R}^k\times f^{-1}(0)$ is an $\e$-Gromov Hausdorff approximation.
\end{lemma}
For a proof of part (a), see e.g., the proof of Lemma 4.11 in \cite{WZ}. Part (b) holds because from Lemma \ref{equiv. splitting maps} there is a $C\delta^{1/2}$-splitting map $u$ on $B_{2R}(p)$, which implies that $B_R(p)$ is $\e$ close in Gromov Hausdorff sense to a ball in $\mb{R}^k\times u^{-1}(0)$ (see e.g. the proof of Proposition 11.1 in \cite{Bam}). Then the conclusion in (b) follows from the fact that $f$ and $u$ are close as shown in Lemma \ref{equiv. splitting maps}.

Next, as in \cite{ChNa2}, to rule out the codimension 2 singularity, we only need to show that $\mb{R}^{n-2}\times C(S^1_{\beta})$, $\beta< 2 \pi$, is not the GH limit of sequences of manifolds under our assumptions. The reason is the following.
When the Ricci curvature is bounded, from Theorem 5.2 in \cite{ChCo2}, if there is a codimension 2 singularity, then a tangent cone is a metric cone $\mb{R}^{n-2}\times Y$ where $Y$ is a cone over a one dimensional compact metric space of diameter $\le \pi$. Here the diameter means the maximum length of minimal geodesics. In our setting, the situation is the same by virtue of Theorem 4.3 in \cite{WZ}.

We argue by contradiction, and assume that there is a sequence of pointed Riemannian manifolds $(\M^n_j, g_j, p_j)$ and smooth functions $L_j\in C^{\infty}(\M_j)$ with $|Ric_{\M_j}+\d^2 L_j|\leq (n-1)\delta_j\rightarrow 0$ and $|\d L_j|\leq \delta_j\rightarrow 0$, and $\vol(B_{10}(p_j))\geq \rho$ satisfying
\[
(\M_j, d_j, p_j)\xrightarrow{d_{GH}}(\mb{R}^{n-2}\times C(S^1_{\beta}),d,p),
\]
where $S^1_{\beta}$ is a circle of circumference $\beta<2\pi$ and $p$ is a vertex of the cone.

Let $\e_j\rightarrow 0$, and $f_j: B_{\e_j^{-1}}(p)\rightarrow B_{\e^{-1}_j}(p_j)$ the $\e_j$-Gromov-Hausdorff approximation. Denote by $\mathcal{S}_j=f_j(\mathcal{S})$. Since away from $\mathcal{S}_j$ the balls in $\M_j$ are close to balls in $\mb{R}^n$ in Gromov-Hausdorff sense, as shown in the proof of Theorem \ref{thm harmonic radius bound}, the $W^{1,q}$ ($q>2n$) harmonic radius $r_h(x)$ is continuous. In particular, we have $r_h(x)\geq \frac{\tau}{2}$ for any $x\in B_1(p_j)\setminus T_{\tau}(\mathcal{S}_j)$.

Then we can choose $\delta_j$ small enough so that there is an $L_j$-drifted $\delta_j$-splitting maps $u_j: B_2(p_j)\rightarrow \mb{R}^{n-2}$ satisfying Theorem \ref{slicing theorem}. Hence, it is possible to pick a $s_j\in G_{\e_j}\cap B_{1/10}(p_j)$ and choose the smallest $W^{1,q}$ harmonic radius on the submanifold $u_j^{-1}(s_j)\cap B_1(p_j)$, namely let
\[
r_j=\min\{r_h(x):~ x\in u_j^{-1}(s_j)\cap B_1(p_j)\}.
\]
Assume that $r_j$ is achieved at some point $x_j$, i.e., $r_j=r_h(x_j)$. Then it is not hard to see that $x_j\rightarrow \mathcal{S}_j\cap B_{1/10}(p_j)$ and $r_j\rightarrow 0$.

By Theorem \ref{transformation theorem}, there is a lower triangular matrix $A_j$ with positive diagonal entries, such that $v_j=A_j\circ(u_j-s_j): B_{2r_j}(x_j)\rightarrow \mb{R}^{n-2}$ is an $L_j$-drifted $\e_j$-splitting map.

Proceeding as in \cite{ChNa2}, by passing to a subsequence, the blow-up sequence $(\M^n_j, r_j^{-1}d_j, x_j)$ $\xrightarrow{d_{GH}}(X,d_X,x)$, where $X$ splits off an $\mathbb{R}^{n-2}$ factor. Moreover, $\tilde{v}_j=r_j^{-1}v_j: B_{2}(x_j)\rightarrow \mb{R}^{n-2}$ is an $L_j$-drifted $\e_j$-splitting map. By the proof of Claim 2 in Theorem \ref{transformation theorem}, one can see that $\tilde{v}_j: B_R(x_j)\rightarrow\mb{R}^{n-2}$ is an $L_j$-drifted $C(n,\rho,R)\e_j$-splitting map for any $R>2$. In particular, it implies that $\tilde{v}_j\rightarrow v$ for some $v: X\rightarrow \mb{R}^{n-2}$. Then Lemma \ref{lem coordinate approximation} implies that
\[
X=\mathbb{R}^{n-2}\times v^{-1}(0^{n-2}).
\]

Since for any $y\in \tilde{v}_j^{-1}(0)$, the $W^{1,q}$ radius $r_h(y)\geq 1$, by Theorem \ref{thm compactness}, we know that $X$ is $C^{\a}\cap W^{1,s}$ in a neighborhood of $v^{-1}(0)$. Hence $X$ is a $C^{\a}\cap W^{1,s}$ manifold with $r_h\geq1$, and $(\M_j, r_j^{-2}g_j, x_j)\xrightarrow{C^{\a}\cap W^{1,s}}(X,g_X,x)$ in Cheeger-Gromov sense. In particular, since the $W^{1,q}$ harmonic radius is continuous, we have $r_h(x)=1$.

On the other hand, the expression of the Ricci curvature tensor in harmonic coordinates is
\be\label{Ricci eqn}
\al
g^{ab}\frac{\partial^2 g_{ij}}{\pa v_a\pa v_b} + Q(\pa g, g)= &-2R_{ij}\\
=& -2(R_{ij}+\d_i\d_j L) + 2\d_i\d_j L,
\eal
\ee
where $\{v_1, v_2,\cdots, v_n\}$ is a local harmonic coordinate chart. Since $|Ric+\d^2L|\leq (n-1)r_j^2\rightarrow 0$ and $|\d L|\leq r_j\rightarrow 0$, and the sequence of metrics $\{r_j^{-2}g_j\}$ is converging in $W^{1,s}$ norm on compact sets, one can see that the limit metric $g_{X}$ is a weak solution of the equation
\[
g^{ab}\frac{\partial^2 g_{ij}}{\pa v_a\pa v_b} + Q(\pa g, g)=0.
\]
Therefore, by the standard elliptic regularity theory, it follows that $g_{X}$ is smooth and Ricci flat, and hence $X$ is a flat manifold since the dimension of $v^{-1}(0^{n-2})$ is 2. Moreover, by volume continuity under Gromov-Hausdorff convergence (see e.g. Theorem 4.10 in \cite{WZ}), we have $\vol(B_r(x))\geq C\rho r^n$ for any $r>0$. Thus, it follows that $X=\mb{R}^n$.

Especially, we have $r_h(x)=\infty$ and contradicts to $r_h(x)=1$. Therefore, the singular set has codimension at least 3.\\

Finally, to rule out the codimension 3 singularity, again we can use a similar argument as in \cite{ChNa2}. One just needs to notice that by Lemma \ref{lem coordinate approximation} and Lemma \ref{equiv. splitting maps}, the $\e$-splitting map $u_j: B_2(p_j)\rightarrow \mb{R}^{n-3}$ still exists. Then since the metrics converge in $C^{\a}$ norm and $u_j$ are harmonic functions, we can still get the bounds on the gradient and hessian of $u_j$. Also, the Poisson approximation $h_j$ of the square of the distance function exists by Lemmas 2.3 and 2.4 in \cite{WZ} (See also the proof of Theorem 6.3 in \cite{ZZ}). Since $\Delta h_j=2n$ and the metrics $g_j$ have uniform $C^{\a}$ bound, the standard elliptic regularity theory implies that $h_j$ have $C^2$ bound.

Therefore, this completes the proof. \qed

\appendix

\section{}

In this section, we prove Claim 4 and finish the proof of Theorem \ref{transformation theorem} part a).\\

{\it Proof of Claim 4:} We first show that
\be\label{claim4 eq2}
\fint^{L_j}_{B_R(x_j)}\left||<w_j^{k-1}, d v_j^k>|^2-\fint^{L_j}_{B_R(x_j)}|<w_j^{k-1}, d v_j^k>|^2\right|\leq \e_j(R).
\ee
Here, as usual, define
\[
<w_j^{k-1}, d v_j^k>=\sum_{a=1}^{k-1}<d v_j^a, dv_j^k>dv_j^1\wedge\cdots\wedge\widehat{dv_j^a}\wedge\cdots\wedge dv_j^{k-1}.
\]
Thus,
\[
\al
&\left|<w_j^{k-1}, dv_j^k>\right|^2\\
=&\sum_{a,b=1}^{k-1}<d v_j^a, dv_j^k> <d v_j^b, dv_j^k> <dv_j^1\wedge\cdots\wedge\widehat{dv_j^a}\wedge\cdots\wedge dv_j^{k-1}, dv_j^1\wedge\cdots\wedge\widehat{dv_j^b}\wedge\cdots\wedge dv_j^{k-1}>.
\eal
\]
Since $(v_j^1, \cdots, v_j^{k-1})$ is an $\e_j(R)$-splitting on $B_R(x_j)$ by Claim 2, it is not hard to see that
\be\label{claim4 eq3}
\al
&\fint^{L_j}_{B_R(x_j)}\bigg|<dv_j^1\wedge\cdots\wedge\widehat{dv_j^a}\wedge\cdots\wedge dv_j^{k-1}, dv_j^1\wedge\cdots\wedge\widehat{dv_j^b}\wedge\cdots\wedge dv_j^{k-1}>\\
&\qquad\quad -\fint^{L_j}_{B_R(x_j)}<dv_j^1\wedge\cdots\wedge\widehat{dv_j^a}\wedge\cdots\wedge dv_j^{k-1}, dv_j^1\wedge\cdots\wedge\widehat{dv_j^b}\wedge\cdots\wedge dv_j^{k-1}>\bigg|\leq \e_j(R).
\eal
\ee
Therefore, \eqref{claim4 eq2} follows from \eqref{claim3 eq2} and \eqref{claim4 eq3} immediately.

Notice that $w_j^k=w_j^{k-1}\wedge dv_j^k$, and
\[
|w_j^k|^2=|w_j^{k-1}|^2|dv_j^k|^2-|<w_j^{k-1},dv_j^k>|^2.
\]
From \eqref{claim3 eq1}, we have
\[
\al
\e_j(R)&\geq\fint^{L_j}_{B_R(x_j)}\left||w^k_j|^2-\fint^{L_j}_{B_R(x_j)}|w^k_j|^2\right|\\
&\geq\fint^{L_j}_{B_R(x_j)}\left||w_j^{k-1}|^2|d v_j^k|^2-\fint^{L_j}_{B_R(x_j)}|w_j^{k-1}|^2|d v_j^k|^2\right|\\
&\quad -\fint^{L_j}_{B_R(x_j)}\left||<w_j^{k-1}, dv_j^k>|^2-\fint^{L_j}_{B_R(x_j)}|<w_j^{k-1}, dv_j^k>|^2\right|
\eal
\]
This, together with \eqref{claim4 eq2}, implies that
\be\label{claim4 eq4}
\fint^{L_j}_{B_R(x_j)}\left||w_j^{k-1}|^2|dv_j^k|^2-\fint^{L_j}_{B_R(x)}|w_j^{k-1}|^2|dv_j^k|^2\right|\leq \e_j(R).
\ee
Since
\[
1-C\e_j\leq \fint^{L_j}_{B_R(x_j)}|w_j^{k-1}|^2\leq 1+C\e_j,
\]
one gets
\[
\al
&\fint^{L_j}_{B_R(x_j)}\left||dv_j^k|^2-\fint^{L_j}_{B_R(x_j)}|dv_j^k|^2\right|\\
\leq & (1+C\e_j)\fint^{L_j}_{B_R(x_j)}|w_j^{k-1}|^2(\tilde{y})d\tilde{y}\fint^{L_j}_{B_R(x_j)}\left||dv_j^k|^2(y)-\fint^{L_j}_{B_R(x_j)}|dv_j^k|^2(z)dz\right|dy\\
\leq& C\fint^{L_j}_{B_R(x_j)}\fint^{L_j}_{B_R(x_j)}\left||w_j^{k-1}|^2(\tilde{y})|dv_j^k|^2(y)-\fint^{L_j}_{B_R(x_j)}|w_j^{k-1}|^2(\tilde{y})|dv_j^k|^2(z)dz\right|dyd\tilde{y}\\
\leq& C \fint^{L_j}_{B_R(x_j)}\fint^{L_j}_{B_R(x_j)}\left||w_j^{k-1}|^2(\tilde{y})|dv_j^k|^2(y)-\fint^{L_j}_{B_R(x_j)}|w_j^{k-1}|^2(z)|dv_j^k|^2(z)dz\right|dyd\tilde{y}\\
&\  +C\fint^{L_j}_{B_R(x_j)}\fint^{L_j}_{B_R(x_j)}\left|\fint^{L_j}_{B_R(x_j)}\left[|w_j^{k-1}|^2(z)-|w_j^{k-1}|^2(\tilde{y})\right]|dv_j^k|^2(z)dz\right|dyd\tilde{y}\\
:=& I+II.
\eal
\]

Using the fact that $|dv_j^k|\leq 1+C\e$,  $\fint^{L_j}_{B_R(x_j)}\left||w_j^{k-1}|^2-1\right|\leq \e_j$ and \eqref{claim4 eq4}, one has
\[
\al
I&\leq C \fint^{L_j}_{B_R(x_j)}\fint^{L_j}_{B_R(x_j)}\Bigg||w_j^{k-1}|^2(\tilde{y})|dv_j^k|^2(y)-|dv_j^k|(y)\\
&\hspace{3.8cm} +|dv_j^k|^2(y)-|w_j^{k-1}|^2(y)|dv_j^k|^2(y)\\
&\hspace{3.8cm} +\left.|w_j^{k-1}|^2(y)|dv_j^k|^2(y)-\fint^{L_j}_{B_R(x_j)}|w_j^{k-1}|^2(z)|dv_j^k|^2(z)dz\right|dyd\tilde{y}\\
&\leq \e_j(R),
\eal
\]
and
\[
\al
II&\leq C\fint^{L_j}_{B_R(x_j)}\fint^{L_j}_{B_R(x_j)}\fint^{L_j}_{B_R(x_j)}\left||w_j^{k-1}|^2(z)-|w_j^{k-1}|^2(\tilde{y})\right|dzdyd\tilde{y}\\
&\leq C\fint^{L_j}_{B_R(x_j)}\left||w_j^{k-1}|^2(z)-1\right|dz+C\fint^{L_j}_{B_R(x_j)}\left||w_j^{k-1}|^2(\tilde{y})-1\right|d\tilde{y}\\
&\leq \e_j(R).
\eal
\]
Therefore, it finishes the proof of Claim 4.\qed

Recall that
\[
a_j^l=-\fint^{L_j}_{B_2(x_j)}<dv_j^l, dv_j^k>\bigg/\fint^{L_j}_{B_2(x_j)}|d v_j^l|^2, \ \forall 1\leq l\leq k-1,
\]
and
\[
\tilde{v}_j^l=v_j^l,\quad \tilde{v}_j^k=v_j^k+\sum_{l=1}^{k-1}a_j^lv_j^l.
\]

Since $(v_j^1,\cdots,v_j^k): B_2(x_j)\rightarrow\mathbb{R}^k$ is $L_j$-drifted $C\e$-splitting, and $(v_j^1,\cdots,v_j^{k-1}): B_2(x_j)\rightarrow\mathbb{R}^{k-1}$ is $L_j$-drifted $\e_j$-splitting, it follows that
\[
\left|\fint^{L_j}_{B_2(x_j)}<dv_j^l, dv_j^k>\right|\leq C\e,\ \textrm{and}\ \fint^{L_j}_{B_2(x_j)}\left||d v_j^l|^2-1\right|\leq\e_j^2.
\]
Hence,
\be\label{ajl}
|a_j^l|\leq C\e.
\ee
The above facts and \eqref{claim3 eq2} imply that
\be\label{main proof eq1}
\al
&\fint^{L_j}_{B_2(x_j)}\left|<d \tilde{v}_j^l, d \tilde{v}_j^k>\right|\\
\leq& \fint^{L_j}_{B_2(x_j)}\left|<d v_j^l, dv_j^k>-\fint^{L_j}_{B_2(x_j)}<dv_j^l, d v_j^k>\right|+ \fint^{L_j}_{B_2(x_j)}\left|a_j^l|dv_j^l|^2-a_j^l\fint^{L_j}_{B_2(x_j)}|dv_j^l|^2\right|\\
&\quad +\sum_{m\neq l}^{k-1}\fint^{L_j}_{B_2(x_j)}|a_j^m||<d v_j^l,dv_j^m>|\\
\leq& \e_j,
\eal
\ee
and
\be\label{main proof eq2}
\al
&\fint^{L_j}_{B_2(x_j)}|d\tilde{v}_j^k|^2\\
=&\fint^{L_j}_{B_2(x_j)}|dv_j^k|^2+\sum_{l=1}^{k-1}|a_j^l|^2\fint^{L_j}_{B_2(x_j)}|dv_j^l|^2\\
&\quad + \sum_{l\neq m}^{k-1}a_j^la_j^m\fint^{L_j}_{B_2(x_j)}<dv_j^l,dv_j^m>+\sum_{l=1}^{k-1}2a_j^l\fint^{L_j}_{B_2(x_j)}<dv_j^l,dv_j^k>\\
\geq& 1-C\e.
\eal
\ee
Setting
\[ \hat{v}_j^l=\tilde{v}_j^l,\quad \hat{v}_j^k=\tilde{v}_j^k\Big/\fint^{L_j}_{B_2(x_j)}|d\tilde{v}_j^k|^2,\]
one has
\[
\al
\fint^{L_j}_{B_2(x_j)}\left||d\hat{v}_j^k|^2-1\right|&=\left(\fint^{L_j}_{B_2(x_j)}|d\tilde{v}_j^k|^2\right)^{-1}\fint^{L_j}_{B_2(x_j)}\left||d\tilde{v}_j^k|^2-\fint^{L_j}_{B_2(x_j)}|d\tilde{v}_j^k|^2\right|\\
&\leq \e_j,
\eal
\]
where we have used \eqref{claim4 eq1}, \eqref{ajl}, and the fact that $(v_j^1, \cdots, v_j^{k-1})$ is an $L_j$-drifted $\e_j$-splitting map.
Therefore, by using the similar technique as in Lemma 3.34 in \cite{ChNa2} (or mean value inequality), one can get
\be\label{main proof eq3}
\sup_{B_1(x_j)}|d \hat{v}_j^k|\leq 1+\e_j,\ \forall 1\leq a\leq k,
\ee
and
\be\label{main proof eq4}
\fint^{L_j}_{B_1(x_j)}|\d^2 \hat{v}_j^k|^2\leq \e_j.
\ee
Finally, \eqref{main proof eq1} and \eqref{main proof eq2} give
\be\label{main proof eq5}
\fint^{L_j}_{B_1(x_j)}|<d\hat{v}_j^a,d\hat{v}_j^k>|\leq \e_j,\ 1\leq a\leq k-1.
\ee
It is obvious that \eqref{main proof eq3}, \eqref{main proof eq4} and \eqref{main proof eq5} together with the fact that $(v_j^1,\cdots,v_j^{k-1})$ is $L_j$-drifted $\e_j$-splitting imply that
$(\hat{v}_j^1,\cdots,\hat{v}_j^k): B_1(x_j)\rightarrow \mb{R}^k$ is an $L$-drifted $\e_j$-splitting map. Since $B_1(x_j)$ in the metric $r_j^{-2}g_j$ is exactly the ball $B_{r_j}(x_j)$ in the metric $g_j$ and $\e_j\rightarrow 0$, this means that before rescaling $(\hat{v}_j^1,\cdots,\hat{v}_j^k): B_{r_j}(x_j)\rightarrow \mb{R}^k$ is $L_j$-drifted $\e$-splitting when $j$ is sufficiently large, which contradicts to the inductive hypothesis that there is no matrix $A$ such that $A\circ u$ is $L_j$-drifted $\e$-splitting on $B_{r_j}(x_j)$.

Hence, this finishes the proof of Theorem \ref{transformation theorem} part a).


\section*{Acknowledgements}  Q.S.Z is grateful to the Simons foundation for its support. M. Z's research is partially supported by NSFC Grant No. 11501206. M. Z would like to thank Prof. Huai-Dong Cao for constant support and his interest in this paper.
Both authors wish to thank Professors M. Anderson, Hong Huang, A. Naber,  L. H. Wang, G.F. Wei,  Z.L. Zhang  and X.H. Zhu  for helpful suggestions. We are very grateful to Dr. Kewei Zhang who pointed out a typo in the statement of Theorem 1.2 in an earlier version.
Prof. Z.L. Zhang kindly informed us that he and Prof. W.S. Jiang are able to extend Cheeger and Naber's result under the condition that the Ricci curvature is $L^p$ with $p>n/2$. In this case the metric is $W^{2, p}$ rather than $W^{1, p}$ in the current paper.


\begin{thebibliography}{00000}


\bibitem[And]{And} Anderson, Michael, {\it Convergence and rigidity of manifolds under Ricci curvature bounds}, 102 (1990), 429-445.

\bibitem[AC]{AC} Anderson, Michael and Cheeger, Jeff, {\it $C^\a$-Compactness for manifolds with Ricci curvature and  injectivity radius bounded below}, J. Diff. Geometry 35 (1992), 265-281.




\bibitem[Bam]{Bam} Bamler, Richard, {\it  Structure theory of singular spaces},  arXiv:1603.05236.






\bibitem[Ch]{Ch} Cheeger, Jeff, {\it Degeneration of Riemannian metrics under Ricci curvature bounds}, Lezioni Fermiane. [Fermi Lectures] Scuola Normale Superiore, Pisa, 2001.

\bibitem[ChCo1]{ChCo1} Cheeger, Jeff; Colding, Tobias H., {\it Lower bounds on Ricci curvature and the almost rigidity of warped products.}, Ann. of Math. (2) 144 (1996), no. 1, 189-237.

\bibitem[ChCo2]{ChCo2} Cheeger, Jeff; Colding, Tobias H., {\it On the structure of spaces with Ricci curvature bounded below. I.}, J. Differential Geom.  46  (1997),  no. 3, 406-480.

\bibitem[ChCo3]{ChCo3}  Cheeger, Jeff; Colding, Tobias H., {\it On the structure of spaces with Ricci curvature bounded below. II.}, J. Differential Geom. 54 (2000), no. 1, 13-35.

\bibitem[ChCo4]{ChCo4}  Cheeger, Jeff; Colding, Tobias H., {\it On the structure of spaces with Ricci curvature bounded below. III.}, J. Differential Geom. 54 (2000), no. 1, 37-74.

\bibitem[CCT]{CCT}  Cheeger, Jeff; Colding, Tobias H.; Tian, Gang, {\it On the singularities of spaces with bounded Ricci curvature.} Geom. Funct. Anal. 12 (2002), no. 5, 873-914.

\bibitem[ChNa1]{ChNa1}  Cheeger, Jeff; Naber, Aaron, {\it Lower bounds on Ricci curvature and quantitative behavior of singular sets.}, Invent. Math. 191 (2013), no. 2, 321-339.

\bibitem[ChNa2]{ChNa2} Cheeger, Jeff; Naber, Aaron, {\it Regularity of Einstein Manifolds and the codimension 4 Conjecture}, Ann. Math. 182 (2015), 1093-1165.

\bibitem[Co]{Co}Colding, Tobias H., {\it Ricci curvature and volume convergence}. Ann. of Math. (2) 145 (1997), no. 3, 477-501.

\bibitem[CoNa]{CoNa} Colding, Tobias Holck; Naber, Aaron, {\it Sharp H\"older continuity of tangent cones for spaces with a lower Ricci curvature bound and applications},  Ann. of Math. (2) 176 (2012), no. 2, 1173-1229.




\bibitem[DWZ]{DWZ} Dai, Xianzhe; Wei, Guofang; Zhang, Zhenlei, {\it Local Sobolev Constant Estimate for Integral Ricci Curvature Bounds}, arXiv:1601.08191.



\bibitem[GT]{GT} Gilbarg, David; Trudinger, Neil S. {\it Elliptic partial differential equations of second order.} Grundlehren der Mathematischen Wissenschaften, Vol. 224. Springer-Verlag, Berlin-New York, 1977. x+401 pp.

\bibitem[HH]{HH} Hebey, E.; Herzlich, M., {\it Harmonic coordinates, harmonic radius and convergence of Riemannian manifolds}, Rend. Mat. Appl. (7) 17 (1997), no. 4, 569-605 (1998)

\bibitem[xLi]{xLi} Li, Xiang-Dong, {\it Liouville theorems for symmetric diffusion operators on complete Riemannian manifolds.} J. Math. Pures Appl. (9) 84 (2005), no. 10, 1295-1361.


\bibitem[Pe]{Pe} Petersen, Peter, {\it Convergence theorems in Riemannian geometry}, Comparison geometry, 167¨C202,
Math. Sci. Res. Inst. Publ., 30, Cambridge Univ. Press, Cambridge, 1997.

\bibitem[PeWe1]{PeWe1} Petersen, Peter; Wei, Guofang, {\it Relative volume comparison with integral curvature bounds}. Geom. Funct. Anal. 7 (1997), no. 6, 1031-1045.

\bibitem[PeWe2]{PeWe2} Petersen, Peter; Wei, Guofang, {\it Analysis and geometry on manifolds with integral Ricci curvature bounds. II}, Trans. Amer. Math. Soc., 353 (2001), 457-478.

\bibitem[Q]{Q} Zhongmin Qian. {\it Estimates for weighted volumes and applications}, Quart. J. Math. Oxford Ser.(2), 48(1997), 235-242.




\bibitem[TZq]{TZq}  Tian, Gang; Zhang, Qi S., {\it A compactness result for Fano manifolds and K\"ahler Ricci flows.} Math. Ann. 362 (2015), no. 3-4, 965-999.


\bibitem[TZz]{TZz} Tian, Gang; Zhang, Zhenlei, {\it Regularity of K\"ahler-Ricci flows on Fano manifolds},  Acta Math. 216 (2016), no. 1, 127-176.



\bibitem[WW]{WW} Wei, Guofang; Wylie, William, {\it Comparison geometry for the Bakry-\'Emery Ricci tensor.} J. Differential Geom. 83 (2009), no. 2, 377-405.

\bibitem[WZ]{WZ} Wang, Feng; Zhu, Xiaohua, {\it Structure of spaces with Bakry-\'Emery Ricci curvature bounded below,}
    arXiv:1304.4490

\bibitem[Ya]{Ya}  Yang, Deane, {\it Convergence of Riemannian manifolds with integral bounds on curvature}, I. Ann. Sci. \'Ecole Norm. Sup. (4) 25 (1992), no. 1, 77-105.


\bibitem[ZZ]{ZZ} Zhang, Qi S.; Zhu, Meng, {\it New volume comparison results and applications to degeneration of Riemannian metrics}, arXiv:1605.09420.


\end{thebibliography}
\end{document}